\documentclass[11pt, a4paper, oneside]{amsart}
\usepackage{vmargin}
\usepackage{amssymb}
\usepackage{url}
\usepackage{hyperref}
\newtheorem{theorem}{Theorem}[section]
\newtheorem{lemma}[theorem]{Lemma}
\newtheorem{proposition}[theorem]{Proposition}

\newtheorem{problem}{Problem}
\newtheorem{remark}{Remark}
\newtheorem{conjecture}{Conjecture}

\numberwithin{equation}{section}

\begin{document}
\title[Rigidity of proper holomorphic maps]{Rigidity of proper holomorphic maps between type-$\mathrm{I}$ irreducible bounded symmetric domains}
\author{Shan Tai Chan}
\address{Department of Mathematics, The University of Hong Kong, Pokfulam Road, Hong Kong}
\curraddr{}
\email{mastchan@hku.hk}
\thanks{}

\subjclass[2010]{Primary 32M15, 32H02, 32L25}

\begin{abstract}
We study proper holomorphic maps between type-$\mathrm{I}$ irreducible bounded symmetric domains. In particular, we obtain rigidity results for such maps under certain assumptions. More precisely, let $f:D^{\mathrm{I}}_{p,q}\to D^{\mathrm{I}}_{p',q'}$ be a proper holomorphic map, where $p\ge q\ge 2$ and $q'<\min\{2q-1,p\}$. Then, we show that $p'\ge p$ and $q'\ge q$.
Moreover, we prove that there exist automorphisms $\psi$ and $\Phi$ of $D^{\mathrm{I}}_{p,q}$ and $D^{\mathrm{I}}_{p',q'}$ respectively, such that $f=\Phi\circ G_h\circ \psi$ for some map
$G_h:D^{\mathrm{I}}_{p,q}\to D^{\mathrm{I}}_{p',q'}$ defined by
$G_h(Z):= \begin{bmatrix} Z & {\bf 0}\\ {\bf 0} & h(Z) \end{bmatrix}$ for all $Z\in D^{\mathrm{I}}_{p,q}$, where $h:D^{\mathrm{I}}_{p,q}\to D^{\mathrm{I}}_{p'-p,q'-q}$ is a holomorphic map.
\end{abstract}

\maketitle

\section{Introduction}
In \cite{Tsai:1993}, Tsai has proven the total geodesy of proper holomorphic maps between irreducible bounded symmetric domains $D$ and $\Omega$ when $\mathrm{rank}(D)=\mathrm{rank}(\Omega)\ge 2$.
In general, there are irreducible bounded symmetric domains $D$ and $\Omega$ with $\mathrm{rank}(\Omega)>\mathrm{rank}(D)\ge 2$ such that some proper holomorphic maps from $D$ to $\Omega$ are nonstandard (i.e., not totally geodesic).
On the other hand, Tu \cite{Tu:2002} and Ng \cite{Ng:2015} have obtained rigidity results on proper holomorphic maps between certain type-$\mathrm{I}$ irreducible bounded symmetric domains of rank $\ge 2$ and of positive rank differences.
Later on, Kim-Zaitsev \cite{Kim_Zaitsev:2015,Kim:2015} have proven that any proper holomorphic map between certain type-$\mathrm{I}$ irreducible bounded symmetric domains of rank $\ge 2$ must be of the specific form up to equivalence under the assumption that such a map extends smoothly around a smooth boundary point of the domain.
Motivated by Kim-Zaitsev \cite{Kim_Zaitsev:2015}, Kim \cite{Kim:2015} and Ng \cite{Ng:2015}, we aim at generalizing the rigidity results of Kim-Zaitsev \cite{Kim_Zaitsev:2015} and Kim \cite{Kim:2015} by removing the smoothness assumption in this article.
This will also generalize some results of Ng \cite{Ng:2015}.
In our study, we will make use of the holomorphic double fibrations, fibral-image-preserving maps between type-$\mathrm{I}$ irreducible bounded symmetric domains and their (local) moduli maps introduced by Ng \cite{Ng:2015}.
Indeed, the study of such (local) moduli maps is particularly important for the classification of proper holomorphic maps between type-$\mathrm{I}$ irreducible bounded symmetric domains (see Seo \cite{Seo:2016}).

Let $r$ and $s$ be positive integers. Denote by $M(r,s;\mathbb C)$ the set of all $r$-by-$s$ complex matrices. For any matrix $Z\in M(r,s;\mathbb C)$, we denote by $Z^T\in M(s,r;\mathbb C)$ the transpose of $Z$. The type-$\mathrm{I}$ irreducible bounded symmetric domains are defined by
\[ D^{\mathrm{I}}_{r,s}:=\left\{Z\in M(r,s;\mathbb C): {\bf I}_s - \overline{Z}^T Z > 0 \right\}. \]
We will simply call them the type-$\mathrm{I}$ domains.
On the other hand, the \emph{generalized complex balls} are defined by
\[ D_{r,s}:= \left\{[z_1,\ldots,z_{r+s}]\in \mathbb P^{r+s-1}:\sum_{j=1}^r |z_j|^2 > \sum_{j=r+1}^{r+s} |z_j|^2 \right\} \]
for positive integers $r$ and $s$.
In what follows, we will write $[A,B]_r\in D_{r,s}$ to indicate that $A\in M(1,r;\mathbb C)$ and $B\in M(1,s;\mathbb C)$ satisfy $[A,B]\in \mathbb P^{r+s-1}$ and $A\overline{A}^T-B\overline{B}^T>0$.
We also write 
\[ [z_1,\ldots,z_{n+1}]=[z_1,\ldots,z_r;z_{r+1},\ldots,z_{n+1}]_{r} \in \mathbb P^{n},\quad 1\le r\le n, \]
for convenience.

Denote by $\mathrm{Aut}(U)$ the automorphism group of a connected complex manifold $U$. Let $F,G:X\to Y$ be holomorphic maps between connected complex manifolds $X$ and $Y$. We say that $F$ and $G$ are \emph{equivalent}, which is written as $F\sim G$, if $F=\Psi\circ G\circ \psi$ for some $\Psi\in \mathrm{Aut}(Y)$ and $\psi\in \mathrm{Aut}(X)$. A holomorphic map $f:D^{\mathrm{I}}_{p,q}\to D^{\mathrm{I}}_{p',q'}$ is said to be \emph{standard} if $f$ is equivalent to the standard embedding $\iota:D^{\mathrm{I}}_{p,q}\hookrightarrow D^{\mathrm{I}}_{p',q'}$ given by $\iota(Z):=\begin{bmatrix} Z & {\bf 0}\\ {\bf 0} & {\bf 0} \end{bmatrix}$. For positive integers $p,q,p'$ and $q'$ such that $p< p'$ and $q< q'$, we define the map $G_{h}:D^{\mathrm{I}}_{p,q}\to D^{\mathrm{I}}_{p',q'}$ by 
\[ G_{h}(Z):= \begin{bmatrix} Z & {\bf 0}\\ {\bf 0} & h(Z)\end{bmatrix}\quad \forall\; Z\in D^{\mathrm{I}}_{p,q}, \]
where $h:D^{\mathrm{I}}_{p,q} \to D^{\mathrm{I}}_{p'-p,q'-q}$ is a holomorphic map. Then, we say that a holomorphic map $f:D^{\mathrm{I}}_{p,q}\to D^{\mathrm{I}}_{p',q'}$ is of \textbf{diagonal type} if and only if $f$ is equivalent to the map $G_{h}$ for some holomorphic map $h:D^{\mathrm{I}}_{p,q} \to D^{\mathrm{I}}_{p'-p,q'-q}$ (cf.\,Ng-Tu-Yin \cite{NTY:2016}). A holomorphic map $f:D^{\mathrm{I}}_{p,q}\to D^{\mathrm{I}}_{p',q'}$ is said to be of \textbf{non-diagonal type} if and only if $f$ is not of diagonal type.
Then, a standard map $f:D^{\mathrm{I}}_{p,q}\to D^{\mathrm{I}}_{p',q'}$ is automatically of diagonal type.

Denote by $\mathcal O(M_1,M_2)$ the space of all holomorphic maps from a complex manifold $M_1$ to a complex manifold $M_2$.
Let $\mathrm{PH}(D^{\mathrm{I}}_{p,q},D^{\mathrm{I}}_{p',q'})$ be the space of all proper holomorphic maps from $D^{\mathrm{I}}_{p,q}$ to $D^{\mathrm{I}}_{p',q'}$, where $p,q,p'$ and $q'$ are positive integers.
Then, for any positive integers $p,q,p'$ and $q'$ satisfying $p<p'$ and $q<q'$, we have the map $\mu_{p,q;p',q'}$ $:$ $\mathcal O(D^{\mathrm{I}}_{p,q},D^{\mathrm{I}}_{p'-p,q'-q})$ $\to$ $\mathrm{PH}(D^{\mathrm{I}}_{p,q},D^{\mathrm{I}}_{p',q'})$ defined by $\mu_{p,q;p',q'}(h):=G_h$, i.e.,
\[ \mu_{p,q;p',q'}(h)(Z)=\begin{bmatrix}
Z & {\bf 0}\\
{\bf 0} & h(Z)
\end{bmatrix}\quad \forall\;Z\in D^{\mathrm{I}}_{p,q}. \]

Motivated by Kim-Zaitsev \cite{Kim_Zaitsev:2015} and Kim \cite{Kim:2015}, we state a conjecture regarding the rigidity of proper holomorphic maps between type-$\mathrm{I}$ irreducible bounded symmetric domains of rank $\ge 2$, as follows.
\begin{conjecture}\label{Conj1}
Let $F:D^{\mathrm{I}}_{p,q}\to D^{\mathrm{I}}_{p',q'}$ be a proper holomorphic map, where $p\ge q \ge 2$ and $q' < p$. Suppose {\rm (1)} $p' < 2p-1$ or {\rm (2)} $q'<2q-1$ holds. Then, we have $p\le p'$, $q\le q'$ and $F$ is of diagonal type, i.e., $F$ is equivalent to the map $G_h$ for some holomorphic map $h:D^{\mathrm{I}}_{p,q} \to D^{\mathrm{I}}_{p'-p,q'-q}$.
\end{conjecture}
\begin{remark}\text{}
\begin{enumerate}
\item In Kim-Zaitsev \cite[Corollary 1]{Kim_Zaitsev:2015} and Kim \cite[Theorem 1.2]{Kim:2015}, Kim and Zaitsev have solved Conjecture \ref{Conj1} under the additional assumption that the map $F$ extends smoothly to a neighborhood of a smooth boundary point of $D^{\mathrm{I}}_{p,q}$.
\item The condition $p'<2p-1$ or $q'<2q-1$ is actually necessary. Note that for any integers $p,q,p'$ and $q'$, we have precisely two mutually exclusive cases
\begin{enumerate}
\item[(a)] $p'<2p-1$ or $q'<2q-1$,
\item[(b)] $p'\ge 2p-1$ and $q'\ge 2q-1$.
\end{enumerate}
For Case {\rm(b)}, we put $p'=2p-1$ and $q'=2q$, Seo \cite[Theorem 1.2]{Seo:2016} has constructed a family of inequivalent proper holomorphic maps from $D^{\mathrm{I}}_{p,q}$ to $D^{\mathrm{I}}_{2p-1,2q}$ for $p\ge 2$, $q\ge 2$, which are not of diagonal type. Actually, it follows from Seo \cite{Seo:2016} that if $p,q,p'$ and $q'$ are positive integers such that $p\ge 2$, $q\ge 2$, $p'\ge 2p-1$ and $q'\ge 2q-1$, then there exists a proper holomorphic map from $D^{\mathrm{I}}_{p,q}$ to $D^{\mathrm{I}}_{p',q'}$ which is not of diagonal type.
\end{enumerate}
\end{remark}

Our main result is a solution to Case (2) of Conjecture \ref{Conj1}, as follows.

\begin{theorem}[Main Theorem]\label{thm:MT}
Let $f:D^{\mathrm{I}}_{p,q}\to D^{\mathrm{I}}_{p',q'}$ be a proper holomorphic map, where $p,q,p'$ and $q'$ are positive integers such that $p\ge q\ge 2$ and $q'<\min\{2q-1,p\}$.
Then, we have $p'\ge p$ and $q'\ge q$.
Moreover, $f$ is of diagonal type.
\end{theorem}

In the proof of Theorem \ref{thm:MT}, the first step is to prove that such a proper holomorphic map $f:D^{\mathrm{I}}_{p,q}\to D^{\mathrm{I}}_{p',q'}$ maps every maximal invariantly geodesic subspace $\Pi\cong D^{\mathrm{I}}_{p,q-1}$ of $D^{\mathrm{I}}_{p,q}$ into some maximal invariantly geodesic subspace $\Pi'\cong D^{\mathrm{I}}_{p',q'-1}$ of $D^{\mathrm{I}}_{p',q'}$. This can be done by making use of the results of Ng \cite{Ng:2015} with slight modification.
We note that Ng \cite[Section 2.4]{Ng:2015} has obtained the explicit parameter space $D^l_{q',p'}$ of the moduli space of all invariantly geodesic subspaces biholomorphic to $ D^{\mathrm{I}}_{p',q'-l}$ in $D^{\mathrm{I}}_{p',q'}$, where $1\le l\le q'-1$.
One of the key ingredients in our proof of Theorem \ref{thm:MT} is to show that $f$ actually has a global meromorphic moduli map $g$ from $D_{q,p}$ into $D_{q',p'}$ by making use of the double fibration for $D^l_{q',p'}$ (see \,\cite[p.\,11]{Ng:2015}).
Another major result is to obtain the rigidity of such kind of meromorphic maps from $D_{q,p}$ into $D_{q',p'}$ when $p\ge q\ge 2$ and $q'<\min\{2q-1,p\}$.
We could then deduce that $f$ is of diagonal type by using the rigidity of its global meromorphic moduli map $g$.

\section{Preliminaries}
We first recall some results of Mok-Tsai \cite{Mok_Tsai:1992} and Tsai \cite{Tsai:1993} on proper holomorphic maps between bounded symmetric domains.
\begin{proposition}[cf.\,Mok-Tsai \cite{Mok_Tsai:1992} and Tsai \cite{Tsai:1993}] \label{Prop:rk_nondecrease1}
Let $F:D\to \Omega$ be a proper holomorphic map between bounded symmetric domains $D$ and $\Omega$. Then, we have $\mathrm{rank}(D)$ $\le$ $\mathrm{rank}(\Omega)$.
\end{proposition}

From Tsai \cite{Tsai:1993}, invariantly geodesic subspaces of $D^{\mathrm{I}}_{p,q}$ are equivalent to images of the standard embeddings from $D^{\mathrm{I}}_{r,s}$ to $D^{\mathrm{I}}_{p,q}$ given by $Z\mapsto \begin{bmatrix} {\bf 0} & {\bf 0} \\ {\bf 0} & Z\end{bmatrix}$ for some integers $1\le r< p$ and $1\le s< q$. As in Ng \cite{Ng:2015}, we call such a subspace a $(r,s)$-subspace of $D^{\mathrm{I}}_{p,q}$.

\begin{proposition}[cf.\,$\text{Ng \cite[Proposition 1.1]{Ng:2015}}$] \label{Prop:MCS_to_MCS}
Let $F:D^{\mathrm{I}}_{p,q}\to D^{\mathrm{I}}_{p',q'}$ be a proper holomorphic map, where $\mathrm{rank}(D^{\mathrm{I}}_{p,q})=\min\{p,q\}\ge 2$. Then, $F$ maps every $(p-1,q-1)$-subspace of $D^{\mathrm{I}}_{p,q}$ into a $(p'-1,q'-1)$-subspace of $D^{\mathrm{I}}_{p',q'}$.
\end{proposition}

Let $f:U \to M(r,s;\mathbb C)$ be a map, where $U\subset M(p,q;\mathbb C)$ is an open subset and $p,q$ are positive integers. Then, we let $f^T:U\to M(s,r;\mathbb C)$ be the map defined by
\[ f^T(Z):=(f(Z))^T\quad \forall\; Z\in U \subset M(p,q;\mathbb C). \]
For any subset $S\subset M(p,q;\mathbb C)$, we let $S^\dagger:=\{W\in M(q,p;\mathbb C): W^T \in S\}$. It is clear that $(S^\dagger)^\dagger = S$. We also define a map $f^\dagger: U^\dagger \subset M(q,p;\mathbb C) \to M(s,r;\mathbb C)$ by
\[ f^\dagger(W):= f^T(W^T)=(f(W^T))^T\quad \forall\;W\in U^\dagger. \]
We remark here that $(D^{\mathrm{I}}_{p,q})^\dagger= D^{\mathrm{I}}_{q,p}$ and $(f^\dagger)^\dagger=f$. Thus, any holomorphic map $F$ $:$ $D^{\mathrm{I}}_{p,q}$ $\to$ $ D^{\mathrm{I}}_{p',q'}$ induces a holomorphic map $F^\dagger: D^{\mathrm{I}}_{q,p}\to D^{\mathrm{I}}_{q',p'}$ given by $F^\dagger(W):=(F(W^T))^T$ for $W\in D^{\mathrm{I}}_{q,p}$. We have the following basic lemma.

\begin{lemma}\label{lem:Trans1}
Let $f_1:D^{\mathrm{I}}_{p,q}\to D^{\mathrm{I}}_{p',q'}$ and $f_2:D^{\mathrm{I}}_{p',q'}\to D^{\mathrm{I}}_{r,s}$ be holomorphic maps. Then, we have $(f_2\circ f_1)^\dagger = f_2^\dagger \circ f_1^\dagger$.
\end{lemma}
\begin{proof}
For any $W\in D^{\mathrm{I}}_{q,p}$, we have
\[\begin{split}
(f_2\circ f_1)^\dagger (W)
=& ((f_2\circ f_1)(W^T))^T
= (f_2(f_1(W^T)))^T\\
=& \big(f_2 \big( (f_1^\dagger(W))^T \big) \big)^T
= f_2^\dagger (f_1^\dagger(W))=(f_2^\dagger\circ f_1^\dagger)(W),
\end{split}\]
i.e., $(f_2\circ f_1)^\dagger = f_2^\dagger\circ f_1^\dagger$.  
\end{proof}

\section{Non-existence of proper holomorphic maps revisited}
In the study of proper holomorphic maps between type-$\mathrm{I}$ irreducible bounded symmetric domains, we have the following natural question.
\begin{problem}
Let $p,q,p'$ and $q'$ be integers such that $2 \le q <  q'\le p' <p$. Is there a proper holomorphic map from $D^{\mathrm{I}}_{p,q}$ to $D^{\mathrm{I}}_{p',q'}$?
{\rm(}Noting that $\mathrm{rank}(D^{\mathrm{I}}_{p,q})$ $=$ $q$ $<$ $q'$ $=$ $\mathrm{rank}(D^{\mathrm{I}}_{p',q'})$.{\rm)}
\end{problem}
\begin{remark}
If $q=q'$, $2\le q=q'\le p'$ and there is a proper holomorphic map $F: D^{\mathrm{I}}_{p,q}\to D^{\mathrm{I}}_{p',q}$, then $F$ is totally geodesic by Tsai \cite{Tsai:1993}. But $F$ induces a totally geodesic embedding from $\mathbb B^{p}$ to $D^{\mathrm{I}}_{p',q}$, which implies that $p\le \max\{p',q\}=p'$. In other words, there does not exist a proper holomorphic map from $D^{\mathrm{I}}_{p,q}$ to $D^{\mathrm{I}}_{p',q'}$ if $2 \le q = q'\le p' <p$.
\end{remark}

Now, we suppose $2\le q<q'\le p'<p$. It is clear that $q \le p-2$. 
We also observe that if $q=p-2$, then we have $q'=p'=p-1$ from the setting so that the question is about the existence of a proper holomorphic map from $D^I_{p,p-2}$ to $D^I_{p-1,p-1}$, where $p\ge 4$.
But this has been solved by Tu \cite[Corollary 1.2]{Tu:2002}, i.e., there does not exist a proper holomorphic map from $D^I_{p,p-2}$ to $D^I_{p-1,p-1}$ when $p\ge 4$.
Now, we only need to consider the case of $q<p-2$ so that $p\ge 5$, and we may assume that $q<p-2$ from now on.

Write $p':=p-l$ and $q':=p-k-l$ for some integers $k\ge 0$ and $l>0$. Here, we have $l:=p-p'$ and $k:=p'-q'$. If $q\ge p-k-2l$, then we have a standard embedding $D^{\mathrm{I}}_{p,p-k-2l} \hookrightarrow D^{\mathrm{I}}_{p,q}$. This induces a proper holomorphic map from $D^{\mathrm{I}}_{p,p-k-2l}$ to $D^{\mathrm{I}}_{p-l,p-k-l}$. But then from Mok \cite[Theorem 3.1]{Mok:2008}, there is a positive integer $N'(k,l)$ such that for $p\ge N'(k,l)$, there does not exist any proper holomorphic map from $D^{\mathrm{I}}_{p,p-k-2l}$ to $D^{\mathrm{I}}_{p-l,p-k-l}$. Hence, we have
\begin{proposition}[cf.\,Mok \cite{Mok:2008}]\label{Pro:NonEx_PHM1}
Let $p$ and $q$ be integers such that $p\ge q \ge 2$. Let $p'$ and $q'$ be positive integers such that $q < q'\le p' <p$. Then, there is a positive integer $N'$ such that for $p\ge N'$ and $q\ge p'+q'-p$, there does not exist any proper holomorphic map from $D^{\mathrm{I}}_{p,q}$ to $D^{\mathrm{I}}_{p',q'}$.
\end{proposition}

On the other hand, from Ng \cite{Ng:2015} we can deduce the following. 
\begin{proposition}\label{Pro:NonEx_PHM2}
Let $p,q,p'$ and $q'$ be integers such that 
\[ 2 \le q \le \min\{q', p'\}\le \max\{q', p'\} <p.\]
If $2q-1 \ge \min\{q', p'\}$, then there does not exist any proper holomorphic map from $D^{\mathrm{I}}_{p,q}$ to $D^{\mathrm{I}}_{p',q'}$.
\end{proposition}
\begin{proof}
Assume the contrary that there is a proper holomorphic map $f:D^{\mathrm{I}}_{p,q}\to D^{\mathrm{I}}_{p',q'}$. Since $D^{\mathrm{I}}_{p',q'}\cong D^{\mathrm{I}}_{q',p'}$, we may assume without loss of generality that $q'\le p'$ so that $q'\le 2q-1$. This induces a proper holomorphic map $F:=\iota\circ f:D^{\mathrm{I}}_{p,q}\to D^{\mathrm{I}}_{p,q'}$, where $\iota:D^{\mathrm{I}}_{p',q'}\hookrightarrow D^{\mathrm{I}}_{p,q'}$ is a standard embedding. Since $q' \le \min\{2q-1,p\}$, it follows directly from Ng \cite[Theorem 1.3]{Ng:2015} that $F$ is standard. Thus, $f$ is also standard so that $f$ induces a standard embedding from $\mathbb B^{p}$ to $D^{\mathrm{I}}_{p',q'}$. But then this implies that $p\le \max\{p',q'\}$, a plain contradiction. Hence, there does not exist such a proper holomorphic map $f$.  
\end{proof}

\section{Holomorphic double fibrations}
\label{Sec:DF_FIPM}
In this section, we recall the construction of holomorphic double fibrations introduced by Ng \cite{Ng:2015} and its application to the study of proper holomorphic maps between type-$\mathrm{I}$ irreducible bounded symmetric domains. In particular, we will show that all proper holomorphic maps between certain type-$\mathrm{I}$ irreducible bounded symmetric domains are fibral-image-preserving maps with respect to such holomorphic double fibrations.

Let $G(p,q)$ be the complex Grassmannian of complex $p$-dimensional linear subspaces of $\mathbb C^{p+q}$. Then, we may identity $D^{\mathrm{I}}_{p,q}$ as an open subset of $G(p,q)$ as
\[ D^{\mathrm{I}}_{p,q} = \{ [{\bf I}_p,Z]_p\in G(p,q): Z\in D^{\mathrm{I}}_{p,q}\}. \]
Here, $[{\bf I}_p,Z]_p\in G(p,q)$ denotes the complex $p$-dimensional linear subspace of $\mathbb C^{p+q}$ spanned by the row vectors of the matrix $\begin{bmatrix}{\bf I}_p,Z \end{bmatrix}\in M(p,p+q;\mathbb C)$. We also identify $Z\in M(p,q;\mathbb C)$ with $[{\bf I}_p,Z]_p\in G(p,q)$ in this article. For positive integers $p$ and $q$, we consider the double fibration
\begin{equation}\label{Eq:HDF1}
D_{p,q} \xleftarrow{\pi_{p,q}^1} \mathbb P^{p-1}\times D^{\mathrm{I}}_{p,q}\xrightarrow{\pi_{p,q}^2} D^{\mathrm{I}}_{p,q},
\end{equation}
where $\pi_{p,q}^1([X],[{\bf I}_p,Z]_p):=[X,XZ]_p$ and $\pi_{p,q}^2$ is the canonical projection onto $D^{\mathrm{I}}_{p,q}$ (cf. \cite{Ng:2015}). For any $[A,B]_p\in D_{p,q}$ we define the fibral image of $[A,B]_p$ by 
\[ [A,B]_p^\sharp := \pi_{p,q}^2\big((\pi_{p,q}^1)^{-1}([A,B]_p)\big).\] Similarly, for any $[{\bf I}_p,Z]_p\in D^{\mathrm{I}}_{p,q}$ we define the fibral image of $[{\bf I}_p,Z]_p$ by 
\[ [{\bf I}_p,Z]_p^\sharp := \pi_{p,q}^1\big((\pi_{p,q}^2)^{-1}([{\bf I}_p,Z]_p)\big).\]
Then, it follows from \cite[Corollary 2.9]{Ng:2015} that for any $[{\bf I}_p,Z]_p\in D^{\mathrm{I}}_{p,q}$, we have
\[ Z^\sharp=[{\bf I}_p,Z]_p^\sharp
= \{[A,AZ]_p\in D_{p,q}: [A]\in \mathbb P^{p-1}\}\cong \mathbb P^{p-1}. \]
More generally, Ng \cite{Ng:2015} has introduced the double fibration
\begin{equation}\label{Eq:HDF2}
\mathbb P^{p+q-1}\cong G(1,p+q-1) \xleftarrow{\hat\pi_{p,q}^1} \mathcal F^{1,p}_{p+q} \xrightarrow{\hat\pi_{p,q}^2} G(p,q),
\end{equation}
where $\mathcal F^{1,p}_{p+q}:=\{(J,K)\in G(1,p+q-1)\times G(p,q) : J\subset K\}$, $\hat\pi_{p,q}^1(J,K)=J$ and $\hat\pi_{p,q}^2(J,K)=K$.
Consider the open subset $\mathcal D^{1,p}_{p,q}\subset \mathcal F^{1,p}_{p+q}$ defined by
\[ \mathcal D^{1,p}_{p,q}:=\{(J,K)\in \mathcal F^{1,p}_{p+q}: H_{p,q}|_K > 0\}, \]
where $H_{p,q}$ is the standard nondegenerate Hermitian form of signature $(p,q)$ on $\mathbb C^{p+q}$, i.e., $p$ eigenvalues of $H_{p,q}$ are $1$ and the other $q$ eigenvalues are $-1$.
Note that Ng \cite[Proposition 2.7]{Ng:2015} has proven that
\[ \mathcal D^{1,p}_{p,q} \cong \mathbb P^{p-1} \times D^{\mathrm{I}}_{p,q} \]
and the restriction of the double fibration (\ref{Eq:HDF2}) to $\mathcal D^{1,p}_{p,q}$ yields the double fibration (\ref{Eq:HDF1}).
We first state a lemma obtained in Ng \cite{Ng:2015}.

\begin{lemma}[cf.\,Ng \cite{Ng:2015}]\label{lem:sp_to_sp0}
Let $F:D^{\mathrm{I}}_{p,q}\to D^{\mathrm{I}}_{p',q'}$ be a holomorphic map, where $p,q,p'$ and $q'$ are positive integers such that $p\ge 2$ and $p'\ge 2$.
Then, $F$ maps every $(p-1,q)$-subspace of $D^{\mathrm{I}}_{p,q}$ into a $(p'-1,q')$-subspace of $D^{\mathrm{I}}_{p',q'}$ if and only if $F$ is a fibral-image-preserving holomorphic map with respect to the double fibrations
\[ D_{p,q} \xleftarrow{\pi_{p,q}^1} \mathbb P^{p-1}\times D^{\mathrm{I}}_{p,q}\xrightarrow{\pi_{p,q}^2} D^{\mathrm{I}}_{p,q}, \]
\[ D_{p',q'} \xleftarrow{\pi_{p',q'}^1} \mathbb P^{p'-1}\times D^{\mathrm{I}}_{p',q'} \xrightarrow{\pi_{p',q'}^2} D^{\mathrm{I}}_{p',q'}, \]
i.e., for any $[A,B]_p\in D_{p,q}$, we have $F([A,B]_p^\sharp) \subset [C,D]_{p'}^\sharp$ for some $[C,D]_{p'}\in D_{p',q'}$.
\end{lemma}
\begin{proof}
Suppose $F$ maps every $(p-1,q)$-subspace of $D^{\mathrm{I}}_{p,q}$ into a $(p'-1,q')$-subspace of $D^{\mathrm{I}}_{p',q'}$.
For any $[A,B]_p\in D_{p,q}$ we have
\[ [A,B]_p^\sharp = \pi_{p,q}^2\big((\pi_{p,q}^1)^{-1}([A,B]_p)\big)
= \{ [{\bf I}_p,Z]_p\in D^{\mathrm{I}}_{p,q}: AZ=B \} \]
by \cite[p.\,12]{Ng:2015}. From Ng \cite{Ng:2015}, for any $[A,B]_p\in D_{p,q}$, the fibral image
\[ [A,B]_p^\sharp=\{ [{\bf I}_p,Z]_p\in D^{\mathrm{I}}_{p,q}: AZ=B \} \]
is a $(p-1,q)$-subspace of $D^{\mathrm{I}}_{p,q}$. Similarly, for any $[C,D]_{p'}\in D_{p',q'}$, the fibral image $[C,D]_{p'}^\sharp$ is a $(p'-1,q')$-subspace of $D^{\mathrm{I}}_{p',q'}$.
From the assumption, for any $[A,B]_p\in D_{p,q}$ we have $F([A,B]_p^\sharp) \subset [C,D]_{p'}^\sharp$ for some $[C,D]_{p'}\in D_{p',q'}$.

Conversely, suppose $F$ is a fibral-image-preserving holomorphic map with respect to the given holomorphic double fibrations.
Note that any $(r-1,s)$-subspace of $D^{\mathrm{I}}_{r,s}$ can be identified as a fibral image $[A,B]_r^\sharp$ for some $[A,B]\in D_{r,s}$, where $r\ge 2$ and $s\ge 1$ are integers.
Thus, the fibral-image-preserving holomorphic map $F$ actually maps every $(p-1,q)$-subspace of $D^{\mathrm{I}}_{p,q}$ into a $(p'-1,q')$-subspace of $D^{\mathrm{I}}_{p',q'}$, as desired.  
\end{proof}
Let $F: D^{\mathrm{I}}_{p,q}\to D^{\mathrm{I}}_{p',q'}$ be a fibral-image-preserving holomorphic map with respect to the double fibrations
\[ D_{p,q} \xleftarrow{\pi_{p,q}^1} \mathbb P^{p-1}\times D^{\mathrm{I}}_{p,q}\xrightarrow{\pi_{p,q}^2} D^{\mathrm{I}}_{p,q}, \]
\[ D_{p',q'} \xleftarrow{\pi_{p',q'}^1} \mathbb P^{p'-1}\times D^{\mathrm{I}}_{p',q'} \xrightarrow{\pi_{p',q'}^2} D^{\mathrm{I}}_{p',q'}. \]
Then, a holomorphic map $g:U\to D_{p',q'}$, where $U\subset D_{p,q}$ is an open subset, is called a \emph{local moduli map} of $F$ if and only if
\[ F([A,B]_p^\sharp) \subset g([A,B]_p)^\sharp \]
for all $[A,B]_p\in U \subset D_{p,q}$.

Now, we look for properties of the maps $G_h$.
Let $p,q,p'$ and $q'$ be integers such that $p'> p \ge 2$ and $q'>q\ge 2$.
Recall that $G_h:D^{\mathrm{I}}_{p,q}\to D^{\mathrm{I}}_{p',q'}$ is the proper holomorphic map defined by 
\[ G_h(Z)=\begin{bmatrix} Z & {\bf 0}\\ {\bf 0} & h(Z) \end{bmatrix}, \]
where $h:D^{\mathrm{I}}_{p,q}\to D^{\mathrm{I}}_{p'-p,q'-q}$ is a holomorphic map. We observe that $G_h$ maps every $(p-1,q)$-subspace of $D^{\mathrm{I}}_{p,q}$ into a $(p'-1,q')$-subspace of $D^{\mathrm{I}}_{p',q'}$, and $G_h$ also maps every $(p,q-1)$-subspace of $D^{\mathrm{I}}_{p,q}$ into a $(p',q'-1)$-subspace of $D^{\mathrm{I}}_{p',q'}$.
By definition, the map $(G_h)^\dagger: D^{\mathrm{I}}_{q,p}\to D^{\mathrm{I}}_{q',p'}$ is given by $(G_h)^\dagger(W)=(G_h(W^T))^T$ for $W\in D^{\mathrm{I}}_{q,p}$.
Then, for any $W\in D^{\mathrm{I}}_{q,p}$ we have
\[ (G_h)^\dagger(W)=\begin{bmatrix} W & {\bf 0}\\ {\bf 0} & (h(W^T))^T\end{bmatrix}= \begin{bmatrix} W & {\bf 0}\\ {\bf 0} & h^\dagger(W)\end{bmatrix}, \]
i.e., $(G_{h})^\dagger=G_{h^\dagger}$.

In addition, we have the following simple observation.
Let $p,q,p'$ and $q'$ be integers such that $p'> p \ge 2$ and $q'>q\ge 2$.
We let $h:D^{\mathrm{I}}_{p,q}\to D^{\mathrm{I}}_{p'-p,q'-q}$ be a holomorphic map.
We can also define proper holomorphic maps from $D^{\mathrm{I}}_{p,q}$ to $D^{\mathrm{I}}_{p',q'}$ by
\[ Z\mapsto \begin{bmatrix} {\bf 0} & h(Z)\\ Z & {\bf 0} \end{bmatrix}, \quad 
Z\mapsto \begin{bmatrix} h(Z) & {\bf 0}\\ {\bf 0} & Z \end{bmatrix},\quad
Z\mapsto \begin{bmatrix} {\bf 0} & Z\\ h(Z) & {\bf 0} \end{bmatrix}
\]
for $Z\in D^{\mathrm{I}}_{p,q}$. We observe that each of the above maps is equivalent to $G_{\hat h}$ for some holomorphic map $\hat h:D^{\mathrm{I}}_{p,q}\to D^{\mathrm{I}}_{p'-p,q'-q}$, i.e., all of the above maps are of diagonal type as well.

The following basic lemma explains the reason why the study of holomorphic maps $F:D^{\mathrm{I}}_{p,q}\to D^{\mathrm{I}}_{p',q'}$ is the same as that of holomorphic maps $F^\dagger:D^{\mathrm{I}}_{q,p}\to D^{\mathrm{I}}_{q',p'}$.

\begin{lemma}\label{lem:trans2}
Let $f:D^{\mathrm{I}}_{p,q}\to D^{\mathrm{I}}_{p',q'}$ be a holomorphic map.
Suppose that $f^\dagger$ $:$ $D^{\mathrm{I}}_{q,p}$ $\to$ $D^{\mathrm{I}}_{q',p'}$ is equivalent to the map $G_h$ for some holomorphic map $h:D^{\mathrm{I}}_{q,p} \to M(q'-q,p'-p)$.
Then, $f$ is equivalent to the map $G_{h^\dagger}$.
In other words, $f$ is of diagonal type if and only if $f^\dagger$ is of diagonal type.
Actually, for any holomorphic maps $F,G:D^{\mathrm{I}}_{p,q}\to D^{\mathrm{I}}_{p',q'}$ we have
\[ F \sim G \iff F^\dagger \sim G^\dagger. \]
\end{lemma}
\begin{proof}
From the assumption, we have
\[ \Phi\circ (f^\dagger\circ \varphi) = G_h \]
for some $\Phi\in \mathrm{Aut}(D^{\mathrm{I}}_{q',p'})$ and $\varphi\in \mathrm{Aut}(D^{\mathrm{I}}_{q,p})$.
Then, we have 
\[ \Phi^\dagger \circ (f\circ \varphi^\dagger)=\Phi^\dagger \circ (f^\dagger\circ \varphi)^\dagger = (\Phi\circ (f^\dagger\circ \varphi))^\dagger= (G_h)^\dagger\]
by Lemma \ref{lem:Trans1} and $(f^\dagger)^\dagger=f$.
In particular, we have
\[ \Phi^\dagger \circ (f\circ \varphi^\dagger) = G_{h^\dagger} \]
by the fact that $(G_h)^\dagger=G_{h^\dagger}$.
Note that $\Phi^\dagger\in \mathrm{Aut}(D^{\mathrm{I}}_{p',q'})$ and $\varphi^\dagger\in \mathrm{Aut}(D^{\mathrm{I}}_{p,q})$. Thus, we have shown that $f$ is equivalent to $G_{h^\dagger}$.
In other words, if $f^\dagger$ is of diagonal type, then so is $f$.
Since $(f^\dagger)^\dagger=f$, we also obtain that if $f=(f^\dagger)^\dagger$ is of diagonal type, then so is $f^\dagger$.
By the above arguments, it follows readily that for any holomorphic maps $F,G:D^{\mathrm{I}}_{p,q}\to D^{\mathrm{I}}_{p',q'}$ we have
\[ F \sim G \iff F^\dagger \sim G^\dagger. \]
 
\end{proof}

We recall a result of Ng \cite{Ng:2015} that is useful for our study.
\begin{theorem}[cf.\,$\text{Ng \cite[Theorem 5.3]{Ng:2015}}$] \label{Thm:thm5.3_Ng15}
Let $F:D^{\mathrm{I}}_{r,s}\to D^{\mathrm{I}}_{r',s'}$ be a proper holomorphic map, where $r\ge r'\ge 2$.
Suppose that $F$ maps every $(r-1,s)$-subspace of $D^{\mathrm{I}}_{r,s}$ into a $(r'-1,s')$-subspace of $D^{\mathrm{I}}_{r',s'}$.
If $F(D^{\mathrm{I}}_{r,s})$ is not contained in a single $(r'-1,s')$-subspace of $D^{\mathrm{I}}_{r',s'}$, then we have $r=r'$, $s\le s'$ and $F$ is the standard embedding, i.e., $F$ is equivalent to the linear map $Z \mapsto \begin{bmatrix} Z ,\; {\bf 0} \end{bmatrix}$ for $Z\in D^{\mathrm{I}}_{r,s}$.
\end{theorem}

Applying the above results from Ng \cite{Ng:2015} and analogous arguments in Ng \cite[Proof of Proposition 5.5]{Ng:2015}, we have
\begin{proposition}\label{prop:sp_to_sp1}
Let $f:D^{\mathrm{I}}_{p,q}\to D^{\mathrm{I}}_{p',q'}$ be a proper holomorphic map, where $p\ge q \ge 2$ and $q' < p$.
Then, $f$ maps every $(p,q-1)$-subspace of $D^{\mathrm{I}}_{p,q}$ into a $(p',q'-1)$-subspace of $D^{\mathrm{I}}_{p',q'}$.
In particular, defining the map $f^\dagger: D^{\mathrm{I}}_{q,p}\to D^{\mathrm{I}}_{q',p'}$ by
$f^\dagger(Z) = (f(Z^T))^T$ for $Z\in D^{\mathrm{I}}_{q,p}$,
$f^\dagger$ is a proper holomorphic map which maps every $(q-1,p)$-subspace of $D^{\mathrm{I}}_{q,p}$ into a $(q'-1,p')$-subspace of $D^{\mathrm{I}}_{q',p'}$.
Moreover, we have $q\le q'$.
If in addition that $q'\le 2q-1$, then $p\le p'$.
In particular, if $q'=q+1$ so that $q+1=q'\le 2q-1$ by $q\ge 2$, then $p\le p'$.
\end{proposition}
\begin{proof}
From Proposition \ref{Prop:rk_nondecrease1} we have $\mathrm{rank}(D^{\mathrm{I}}_{p',q'})=\min\{p',q'\}\ge q =\mathrm{rank}(D^{\mathrm{I}}_{p,q})$ $\ge$ $2$ so that $p'\ge q\ge 2$ and $q'\ge q\ge 2$.
Let $X_{p,q-1}\subset D^{\mathrm{I}}_{p,q}$ be an arbitrary $(p,q-1)$-subspace.
Then, for any $(p-1,q-1)$-subspace $X'_{p-1,q-1}\subset X_{p,q-1}$, we have $f(X'_{p-1,q-1})\subset Y'_{p'-1,q'-1}$ for some $(p'-1,q'-1)$-subspace $Y'_{p'-1,q'-1}\subset D^{\mathrm{I}}_{p',q'}$ by Proposition \ref{Prop:MCS_to_MCS}.
Thus, we have
\[ f(X'_{p-1,q-1})\subset Y'_{p'-1,q'-1} \subset Y_{p',q'-1} \subset D^{\mathrm{I}}_{p',q'} \]
for some $(p',q'-1)$-subspace $Y_{p',q'-1}$ of $D^{\mathrm{I}}_{p',q'}$.
This induces a proper holomorphic map
\[ f|_{X_{p,q-1}}: X_{p,q-1}\cong D^{\mathrm{I}}_{p,q-1} \to D^{\mathrm{I}}_{p',q'} \]
which maps every $(p-1,q-1)$-subspace of $X_{p,q-1}\cong D^{\mathrm{I}}_{p,q-1}$ into a $(p',q'-1)$-subspace of $D^{\mathrm{I}}_{p',q'}$.
We thus get a proper holomorphic map
\[ (f|_{X_{p,q-1}})^T: X_{p,q-1}\cong D^{\mathrm{I}}_{p,q-1} \to D^{\mathrm{I}}_{q',p'} \]
which maps every $(p-1,q-1)$-subspace of $X_{p,q-1}\cong D^{\mathrm{I}}_{p,q-1}$ into a $(q'-1,p')$-subspace of $D^{\mathrm{I}}_{q',p'}$.
Since $p>q'$, it follows from Ng \cite[Theorem 5.3]{Ng:2015} (i.e., Theorem \ref{Thm:thm5.3_Ng15}) that
\[ f^T(X_{p,q-1}) \subset Y_{q'-1,p'} \]
for some $(q'-1,p')$-subspace $Y_{q'-1,p'}$ of $D^{\mathrm{I}}_{q',p'}$, i.e., $f(X_{p,q-1}) \subset Z_{p',q'-1}$ for some $(p',q'-1)$-subspace $Z_{p',q'-1}$ of $D^{\mathrm{I}}_{p',q'}$.
This shows that $f$ maps every $(p,q-1)$-subspace of $D^{\mathrm{I}}_{p,q}$ into a $(p',q'-1)$-subspace of $D^{\mathrm{I}}_{p',q'}$.
Defining $f^\dagger$ as in the statement of Proposition \ref{prop:sp_to_sp1}, we then get the desired result for $f^\dagger$.

Suppose $q'\le 2q-1$. We are going to prove that $p\le p'$. Assume the contrary that $p'<p$. By composing $f$ with a standard embedding $\iota:D^{\mathrm{I}}_{p',q'}\hookrightarrow D^{\mathrm{I}}_{p,q'}$, we have a proper holomorphic map $\iota\circ f:D^{\mathrm{I}}_{p,q}\to D^{\mathrm{I}}_{p,q'}$.
Since $q'<p$, we have $q' \le \min\{2q-1,p\}$.
Then, it follows from Ng \cite[Theorem 1.3]{Ng:2015} that $\iota\circ f$ is standard, and so is $f$.
But then $f$ would induce a standard embedding from $\mathbb B^p$ to $D^{\mathrm{I}}_{p',q'}$ so that $p\le \max\{p',q'\}$, which contradicts with the assumption that $p'<p$ and $q'<p$.
Hence, we have $p'\ge p$.

Since $q\ge 2$, if $q'=q+1$, then $q'\le 2q-1$ and we still have $p'\le p$ by the above conclusion.  
\end{proof}

As a consequence of Proposition \ref{prop:sp_to_sp1} and the known results from \cite{Tsai:1993,Ng:2015}, we have 

\begin{proposition}\label{prop:proper2}
Let $f:D^{\mathrm{I}}_{p,q}\to D^{\mathrm{I}}_{p',q'}$ be a proper holomorphic map, where $p\ge q \ge 2$ and
$q' < \min\{p,2q-1\}$.
Then, we have 
\begin{enumerate}
\item[(1)] $q\le q'$, $p\le p'$ and $f$ maps every $(p,q-1)$-subspace of $D^{\mathrm{I}}_{p,q}$ into a $(p',q'-1)$-subspace of $D^{\mathrm{I}}_{p',q'}$.
\item[(2)] If $q=q'$ or $p=p'$, then $f$ is standard. In particular, if $f$ is nonstandard, then we have $q<q'$ and $p<p'$.
\end{enumerate}
\end{proposition}
\begin{proof}
Part (1) follows directly from Proposition \ref{prop:sp_to_sp1}.
For Part (2), we first consider the case of $q=q'$.
Since $\mathrm{rank}(D^{\mathrm{I}}_{p,q})=q=q' \ge \mathrm{rank}(D^{\mathrm{I}}_{p',q'})$ and $q\ge 2$, it follows from Tsai \cite[Main Theorem]{Tsai:1993} that $f$ is standard.
Now, if $p=p'$, then $f:D^{\mathrm{I}}_{p,q}\to D^{\mathrm{I}}_{p,q'}$ is a proper holomorphic map with $p\ge q \ge 2$ and
$q' < \min\{p,2q-1\}$ and thus $f$ is standard by Ng \cite[Theorem 1.3]{Ng:2015}. The proof is complete.  
\end{proof}

Concerning \cite[Conjecture 3.10]{NTY:2016}, it is natural to raise the following question.
\begin{problem}\label{Pro:FIP1}
Let $F:D^{\mathrm{I}}_{p,q} \to D^{\mathrm{I}}_{p',q'}$ be a proper holomorphic map, where $p\ge q\ge 2$.
Suppose $p\le \min\{p',q'\}$. Does $F$ still map every $(p,q-1)$-subspace {\rm(}or $(p-1,q)$-subspace{\rm)} into a $(p',q'-1)$-subspace {\rm(}or a $(p'-1,q')$-subspace{\rm)}?
\end{problem}
\begin{remark}
Note that \cite[Conjecture 3.10]{NTY:2016} holds for proper holomorphic maps $F:D^{\mathrm{I}}_{p,q} \to D^{\mathrm{I}}_{p',q'}$ with $p\ge q\ge 2$ and $q'<p$ {\rm(}see Proposition \ref{prop:sp_to_sp1}{\rm)}.
That is the reason why we would like to consider the case where $\min\{p',q'\}\ge p$ in order to solve \cite[Conjecture 3.10]{NTY:2016} completely for all proper holomorphic maps between type-$\mathrm{I}$ irreducible bounded symmetric domains.
\end{remark}

In \cite{Ng:2015}, Ng has proven the following.
\begin{proposition}[cf.\,$\text{Ng \cite[Proposition 2.15]{Ng:2015}}$]\label{Prop:Pro2.15_Ng15}
Let $F:D^{\mathrm{I}}_{p,q} \to D^{\mathrm{I}}_{p',q'}$ be a fibral-image-preserving holomorphic map with respect to the double fibrations
\[ D_{p,q} \xleftarrow{\pi_{p,q}^1} \mathbb P^{p-1}\times D^{\mathrm{I}}_{p,q}\xrightarrow{\pi_{p,q}^2} D^{\mathrm{I}}_{p,q}, \]
 \[ D_{p',q'} \xleftarrow{\pi_{p',q'}^1} \mathbb P^{p'-1}\times D^{\mathrm{I}}_{p',q'} \xrightarrow{\pi_{p',q'}^2} D^{\mathrm{I}}_{p',q'}. \]
Then, $F$ has a local moduli map.
\end{proposition}

Then, by Proposition \ref{Prop:Pro2.15_Ng15} and Proposition \ref{prop:sp_to_sp1} we have

\begin{proposition}\label{Pro:FIP1}
Let $f:D^{\mathrm{I}}_{p,q}\to D^{\mathrm{I}}_{p',q'}$ be a proper holomorphic map, where $p\ge q \ge 2$ and $q' < p$.
Let $f^\dagger: D^{\mathrm{I}}_{q,p}\to D^{\mathrm{I}}_{q',p'}$ be the proper holomorphic map defined by
$f^\dagger(Z)$ $:=$ $(f(Z^T))^T$ for $Z\in D^{\mathrm{I}}_{q,p}$.
Then, $f^\dagger$ is a fibral-image-preserving holomorphic map with respect to the double fibrations
\[ D_{q,p} \xleftarrow{\pi_{q,p}^1} \mathbb P^{q-1}\times D^{\mathrm{I}}_{q,p}\xrightarrow{\pi_{q,p}^2} D^{\mathrm{I}}_{q,p}, \]
 \[ D_{q',p'} \xleftarrow{\pi_{q',p'}^1} \mathbb P^{q'-1}\times D^{\mathrm{I}}_{q',p'} \xrightarrow{\pi_{q',p'}^2} D^{\mathrm{I}}_{q',p'}. \]
In particular, $f^\dagger$ has a local moduli map $g:U\subset D_{q,p} \to D_{q',p'}$ such that $g$ is holomorphic, where $U\subset D_{q,p}$ is a connected open subset.
\end{proposition}
\begin{proof}
By Proposition \ref{prop:sp_to_sp1}, $f^\dagger$ is a proper holomorphic map which maps every $(q-1,p)$-subspace of $D^{\mathrm{I}}_{q,p}$ into a $(q'-1,p')$-subspace of $D^{\mathrm{I}}_{q',p'}$.
Then, Lemma \ref{lem:sp_to_sp0} asserts that $f^\dagger$ is a fibral-image-preserving holomorphic map with respect to the given double fibrations.
The result then follows from Proposition \ref{Prop:Pro2.15_Ng15}.  
\end{proof}

\section{On proper holomorphic maps between generalized complex balls}
In this section, we recall some rigidity results for proper holomorphic maps between generalized complex balls obtained from \cite{Baouendi_Huang:2005,Baouendi_Ebenfelt_Huang:2011,Ng:2013}.
In \cite{Baouendi_Huang:2005}, Baouendi and Huang have obtained the following rigidity theorem for proper holomorphic maps between certain generalized complex balls.
\begin{theorem}[cf.\,$\text{Baouendi-Huang \cite[Theorems 1.1 and 1.4]{Baouendi_Huang:2005}}$]\label{thm:thm1.1&1.4_BH05}
Let $p\in \partial D_{r,s}$ and $U_p$ be a neighborhood of $p$ in $\mathbb P^{r+s-1}$ with $U_p\cap D_{r,s}$ connected. Let $F:U_p\cap D_{r,s}\to D_{r',s'}$ be a holomorphic map. Suppose that for any sequence $\{Z_j\}_{j=1}^{+\infty} \subset U_p \cap D_{r,s}$ with $\lim_{j\to +\infty} Z_j$ $\in$ $\partial D_{r,s}$, all the limit points of the sequence $\{F(Z_j)\}_{j=1}^{+\infty}$ lie in $\partial D_{r',s'}$. If {\rm(1)} $r'\ge r\ge 2$ and $s=s'\ge 2$ or {\rm(2)} $s'\ge s\ge 2$ and $r=r'\ge 2$, then $F$ extends to a totally geodesic embedding from $D_{r,s}$ to $D_{r',s'}$ and $F$ is linear, i.e., $F$ is equivalent to the map
\[ [z_1,\ldots,z_r;z_{r+1},\ldots,z_{r+s}]_r\mapsto
[z_1,\ldots,z_r,{\bf 0}_{r'-r};z_{r+1},\ldots,z_{r+s},{\bf 0}_{s'-s}]_{r'}. \]
\end{theorem}
In addition, the rigidity results in \cite{Baouendi_Ebenfelt_Huang:2011} for rational proper maps between generalized complex balls are actually applicable to our study of proper holomorphic maps between type-$\mathrm{I}$ bounded symmetric domains. Now, we are going to reinterpret these results in \cite{Baouendi_Ebenfelt_Huang:2011}.

Let $F:D_{q,p} \dashrightarrow D_{q',p'}$ be a rational proper map, where $2\le q\le p$, $2\le q'\le p'$, $q'< p$ and $q'<2q-1$.
(Noting that $F$ is only defined on a certain dense open subset $\mathcal U$ of $D_{q,p}$.)
In what follows, for any rational proper map $f:D_{r,s} \dashrightarrow D_{r',s'}$, we write $f:\mathcal U' \subset D_{r,s} \to D_{r',s'}$ to indicate that $\mathcal U' \subset D_{r,s}$ is the domain of the map $f$ in $D_{r,s}$.
Actually, such a map $f$ extends to a rational map from $\mathbb P^{r+s-1}$ to $\mathbb P^{r'+s'-1}$.
Then, there is a point $b \in \partial D_{q,p}$ and an open neighborhood $U$ of $b$ in $\mathbb P^{q+p-1}$ such that
\[ F(U\cap D_{q,p})\subset D_{q',p'},\quad F(U\cap \partial D_{q,p})\subset \partial D_{q',p'}. \]

For integers $l$ and $n$ such that $0\le l\le n-1$, we define the generalized Siegel upper-half space by
\[ \mathbb S^n_l
:=\left\{(z_1,\ldots,z_{n-1},w)\in \mathbb C^{n}: \mathrm{Im}w > -\sum_{j=1}^l |z_j|^2 + \sum_{j=l+1}^{n-1}|z_j|^2\right\} \]
and its boundary is given by
\[ \mathbb H^n_l
:=\left\{(z_1,\ldots,z_{n-1},w)\in \mathbb C^{n}: \mathrm{Im}w = -\sum_{j=1}^l |z_j|^2 + \sum_{j=l+1}^{n-1}|z_j|^2\right\}. \]
Note that from Baouendi-Huang \cite[p.\,380]{Baouendi_Huang:2005}, $\mathrm{Aut}(D_{q,p})$ acts transitively on the boundary $\partial D_{q,p}$, and $U(q,p)$ acts on $D_{q,p}$ as a group of automorphisms, where $U(q,p):=\left\{M\in GL(q+p,\mathbb C): M {\bf I}_{q,p} \overline{M}^T={\bf I}_{q,p}\right\}$ and ${\bf I}_{q,p}:=\begin{bmatrix}
{\bf I}_q & {\bf 0}\\ {\bf 0} & -{\bf I}_p
\end{bmatrix}$.
Thus, from Baouendi-Ebenfelt-Huang \cite{Baouendi_Ebenfelt_Huang:2011} there is a holomorphic map $\Psi:\mathbb C^{q+p-1}\to \mathbb P^{q+p-1}$ such that $\Psi$ maps $\mathbb H^{q+p-1}_{q-1}$ biholomorphically to 
\[ \partial D_{q,p} \smallsetminus \{[z_1,\ldots,z_{p+q}]\in \mathbb P^{p+q-1}: z_1+z_{p+q}=0\}\]
and maps $\mathbb S^{q+p-1}_{q-1}$ biholomorphically to 
\[ D_{q,p} \smallsetminus \{[z_1,\ldots,z_{p+q}]\in \mathbb P^{p+q-1}: z_1+z_{p+q}=0\} \]
such that $0\in \mathbb H^{q+p-1}_{q-1}$ is mapped to $b=\Psi(0)\in \partial D_{q,p}$.
We may assume without loss of generality that $b=[\sqrt{-1},0,\ldots,0,\sqrt{-1}]$ and we define
\[ \Psi(z,w)=[\sqrt{-1}+w,2z,\sqrt{-1}-w]. \]
Similarly, there is a holomorphic map $\Phi:\mathbb C^{q'+p'-1}\to \mathbb P^{q'+p'-1}$ which maps $\mathbb H^{q'+p'-1}_{q'-1}$ biholomorphically to $$\partial D_{q',p'} \smallsetminus \{[z_1,\ldots,z_{p'+q'}]\in \mathbb P^{p'+q'-1}: z_1+z_{p'+q'}=0\}$$ and maps $\mathbb S^{q'+p'-1}_{q'-1}$ biholomorphically to $$D_{q',p'} \smallsetminus \{[z_1,\ldots,z_{p'+q'}]\in \mathbb P^{p'+q'-1}: z_1+z_{p'+q'}=0\}$$ such that $F(b)=\Phi(0)$.
Here, we may also assume $F(b)$ $=$ [$\sqrt{-1}$, $0$,$\ldots$,$0$,$\sqrt{-1}$] and we let
\[ \Phi(z',w')=[\sqrt{-1}+w',2z',\sqrt{-1}-w']. \]
We identify $U\cap D_{q,p}$ with $\widetilde U\cap \mathbb S^{q+p-1}_{q-1}$ for some open neighborhood $\widetilde U$ of $0$, and $F|_{U\cap D_{q,p}}$ is equivalent to $\widetilde F:\widetilde U\cap \mathbb S^{q+p-1}_{q-1}\to \mathbb S^{q'+p'-1}_{q'-1}$ with
\[ \widetilde F(\zeta,w)=(\widetilde f(\zeta,w),g(\zeta,w)), \]
where $\zeta=(\zeta_1,\ldots,\zeta_{q+p-2})$ and $\widetilde F$ is defined on $\widetilde U\cap \mathbb S^{q+p-1}_{q-1}$ with $\Phi\circ \widetilde F=F\circ \Psi|_{\widetilde U\cap \mathbb S^{q+p-1}_{q-1}}$.
Here, $g:\widetilde U\cap \mathbb S^{q+p-1}_{q-1} \to \mathbb C$.
In particular, $\Phi(\widetilde F(0)) = F(b)$ so that $\widetilde F(0)=0$.
We refer the readers to Baouendi-Ebenfelt-Huang \cite[Sections 2 and 3]{Baouendi_Ebenfelt_Huang:2011} for details about such mappings which map some open subset of $\mathbb H^n_l$ into $\mathbb H^N_{l'}$.
For $(z_1,\ldots,z_{n-1},w)\in \mathbb C^n$, we also write
\[ (z_1,\ldots,z_{n-1},w)=(z_1,\ldots,z_{l};z_{l+1},\ldots,z_{n-1},w)_l \]
for convenience.
Then, from Baouendi-Ebenfelt-Huang \cite[Theorems 1.1 and 1.3]{Baouendi_Ebenfelt_Huang:2011} and \cite[Proof of Theorem 1.1 (a), p.\,1656]{Baouendi_Ebenfelt_Huang:2011}, we have the following possibilities.
\begin{enumerate}
\item[(1)] ${\partial g\over \partial w}(0)>0$ and thus $\widetilde F(z,w)$ is equivalent to the map
\[ (\zeta,w)\mapsto (\zeta_1,\ldots,\zeta_{q-1},\psi(\zeta,w);\zeta_{q},\ldots,\zeta_{p+q-2},\psi(\zeta,w),{\bf 0},w)_{q'-1}, \]
where $\psi=(\psi_1,\ldots,\psi_{q'-q})$ if $q'-q \le p'-p$, or
\[ (\zeta,w)\mapsto (\zeta_1,\ldots,\zeta_{q-1},\phi(\zeta,w),{\bf 0};\zeta_{q},\ldots,\zeta_{p+q-2},\phi(\zeta,w),w)_{q'-1}, \]
where $\phi=(\phi_1,\ldots,\phi_{p'-p})$ if $q'-q \ge p'-p$.
\item[(2)] $g\equiv 0$.
\end{enumerate}

\noindent For $[z_1,\ldots,z_{p+q}]\in \mathbb P^{p+q-1}$ such that $z_1+z_{p+q}\neq 0$, we can write
\[ \begin{split}
&[z_1,\ldots,z_{p+q}]\\
=&\left[ {2\sqrt{-1}\over z_1+z_{p+q}} z_1,\ldots,{2\sqrt{-1}\over z_1+z_{p+q}} z_{p+q}\right]\\
=& \left[ \sqrt{-1}+{\sqrt{-1}(z_1-z_{p+q})\over z_1+z_{p+q}},
{2\sqrt{-1} \over z_1+z_{p+q}}{\bf z}',
\sqrt{-1}-{\sqrt{-1}(z_1-z_{p+q})\over z_1+z_{p+q}}\right],
\end{split} \]
where ${\bf z}':=(z_2,\ldots,z_{p+q-1})$.
We have a map 
\[ \Psi^{-1}:\mathbb P^{p+q-1} \smallsetminus \{[z_1,\ldots,z_{p+q}]\in \mathbb P^{p+q-1}:z_1+z_{p+q}= 0\}\to \mathbb C^n\]
given by
\[ \Psi^{-1}([z_1,\ldots,z_{p+q}]):=
\left({\sqrt{-1}z_2\over z_1+z_{p+q}},
\ldots,
{\sqrt{-1}z_{p+q-1}\over z_1+z_{p+q}},
{\sqrt{-1}(z_1-z_{p+q})\over z_1+z_{p+q}}\right) \]
which maps $D_{q,p}\smallsetminus \{[z_1,\ldots,z_{p+q}]\in \mathbb P^{p+q-1}:z_1+z_{p+q}= 0\}$ into $\mathbb S^{p+q-1}_{q-1}$.
By making use of the above transformation, we have
\[ F|_{U\cap D_{q,p}}=\Phi\circ \widetilde F\circ \Psi^{-1}|_{U\cap D_{q,p}} \]
and Case (1) in the above would imply that $F|_{U\cap D_{q,p}}$ is equivalent to the map
\[ [z_1,\ldots,z_{p+q}]\mapsto
[z_1,\ldots,z_q,\phi(z);z_{q+1},\ldots,z_{p+q},\phi(z),{\bf 0}]_{q'}, \]
{\rm(}resp.\,$[z_1,\ldots,z_{p+q}]\mapsto
[z_1,\ldots,z_q,\phi(z),{\bf 0};z_{q+1},\ldots,z_{p+q},\phi(z)]_{q'}${\rm)} if $q'-q \le p'-p$ {\rm(}resp.\,$q'-q$ $\ge$ $p'-p${\rm)},
where $\phi=(\phi_1,\ldots,\phi_{\min\{p'-p,q'-q\}})$ and $z=(z_1,\ldots,z_{p+q})$.
In this case, $F$ is actually equivalent to the map
\[ [z_1,\ldots,z_{p+q}]\mapsto
[z_1,\ldots,z_q,\phi(z);z_{q+1},\ldots,z_{p+q},\phi(z),{\bf 0}]_{q'}, \]
{\rm(}resp.\,$[z_1,\ldots,z_{p+q}]\mapsto
[z_1,\ldots,z_q,\phi(z),{\bf 0};z_{q+1},\ldots,z_{p+q},\phi(z)]_{q'}${\rm)} if $q'-q \le p'-p$ {\rm(}resp.\,$q'-q$ $\ge$ $p'-p${\rm)} by the identity theorem for meromorphic maps.

Case (2) in the above implies that $F|_{U\cap D_{q,p}}$ is of the form
\[ [h,\phi_1,\ldots,\phi_{q'-1};\psi_1,\ldots,\psi_{p'-1},h]_{q'}\]
for some rational functions $h$, $\phi_j$, $1\le j\le q'-1$, and $\psi_i$, $1\le i\le p'-1$.
Thus, $F$ is of the above form by the identity theorem for meromorphic maps.
In particular, $\sum_{j=1}^{q'-1}|\phi_j|^2> \sum_{i=1}^{p'-1}|\psi_i|^2$ and $F$ induces a rational proper map
\[ [\phi_1,\ldots,\phi_{q'-1};\psi_1,\ldots,\psi_{p'-1}]_{q'-1}: \mathcal U\subset D_{q,p}\to D_{q'-1,p'-1}. \]

In conclusion, by reinterpreting Theorems 1.1 and 1.3 in Baouendi-Ebenfelt-Huang \cite{Baouendi_Ebenfelt_Huang:2011}, we have the following result for rational proper maps between certain generalized complex balls.
\begin{theorem}[cf. Baouendi-Ebenfelt-Huang \cite{Baouendi_Ebenfelt_Huang:2011}]\label{thm:BEH11_MainThm}
Let $p,q,p'$ and $q'$ be positive integers such that $2\le q\le p$, $2\le q'\le p'$, $q'< p$ and $q'<2q-1$.
Let $F:U\subset D_{q,p} \to D_{q',p'}$ be a rational proper map, where $U$ is the largest open subset of $D_{q,p}$ such that $F$ is well-defined.
Then, we have the following possibilities.
\begin{enumerate}
\item[(1)] $F$ is equivalent to the map
\[ [z_1,\ldots,z_{p+q}]\mapsto
[z_1,\ldots,z_q,\phi(z);z_{q+1},\ldots,z_{p+q},\phi(z),{\bf 0}]_{q'}, \]
{\rm(}resp.\,$[z_1,\ldots,z_{p+q}]\mapsto
[z_1,\ldots,z_q,\phi(z),{\bf 0};z_{q+1},\ldots,z_{p+q},\phi(z)]_{q'}${\rm)} if $q'-q \le p'-p$ {\rm(}resp.\,$q'-q\ge p'-p${\rm)},
where $\phi=(\phi_1,\ldots,\phi_{\min\{p'-p,q'-q\}})$ such that $\phi_j(z)={p_j(z)\over q_j(z)}$ for some homogeneous polynomials $p_j(z)$ and $q_j(z)$ in $z=(z_1,\ldots,z_{p+q})$ with $\deg p_j$ $=$ $\deg q_j+1$, $1\le j\le \min\{p'-p,q'-q\}$,
\item[(2)] $F$ is of the form
\[ [h,\phi_1,\ldots,\phi_{q'-1};\psi_1,\ldots,\psi_{p'-1},h]_{q'}\]
for some rational functions $h$, $\phi_j$, $1\le j\le q'-1$, and $\psi_i$, $1\le i\le p'-1$. As a consequence, this induces a rational proper map
\[ [\phi_1,\ldots,\phi_{q'-1};\psi_1,\ldots,\psi_{p'-1}]_{q'-1}: U'\subset D_{q,p}\to D_{q'-1,p'-1}. \]
\end{enumerate}
\end{theorem}

On the other hand, Ng \cite{Ng:2013} obtained the following rigidity theorem for proper holomorphic maps between certain generalized complex balls.

\begin{theorem}[cf.\,Ng $\text{\cite[Main Theorem]{Ng:2013}}$]
\label{thm:Main_Theorem_Ng13}
Let $f:D_{r,s}\to D_{r',s'}$ be a proper holomorphic map, where $2\le r\le s$, $2\le r'\le s'$ and $r'<2r-1$.
Then, $f$ extends to a linear embedding from $\mathbb P^{r+s-1}$ into $\mathbb P^{r'+s'-1}$, and thus $f$ is equivalent to the map
\[ D_{r,s}\ni[z_1,\ldots,z_r,w_1,\ldots,w_s]\mapsto
[z_1,\ldots,z_r,{\bf 0};w_1,\ldots,w_s,{\bf 0}]_{r'}\in D_{r',s'}, \]
$r\le r'$ and $s\le s'$.
\end{theorem}

We also have the following rigidity result from Gao-Ng \cite{Gao_Ng:2018} and Baouendi-Huang \cite{Baouendi_Huang:2005} for degree-one rational proper maps between generalized complex balls.

\begin{proposition}\label{pro:linear_rational_proper}
Let $g:U\subset D_{r,s} \to D_{r',s'}$ be a rational proper map.
Suppose $\deg (g)=1$ as a rational map from $\mathbb P^{r+s-1}$ to $\mathbb P^{r'+s'-1}$.
Then, $r\le r'$, $s\le s'$ and $g$ is actually a linear embedding and $g$ is equivalent to the map $\widetilde g$ given by
\[ D_{r,s}\ni[z_1,\ldots,z_r,w_1,\ldots,w_s]\mapsto
[z_1,\ldots,z_r,{\bf 0};w_1,\ldots,w_s,{\bf 0}]_{r'}\in D_{r',s'}, \]
i.e., $g= \Psi\circ \widetilde g \circ \psi$ for some $\psi\in \mathrm{Aut}(D_{r,s})$ and $\Psi\in \mathrm{Aut}(D_{r',s'})$.
\end{proposition}
\begin{proof}
Since $\deg(g)=1$, we may write 
\[ g=[g_1,\ldots,g_{r'};g_{r'+1},\ldots,g_{r'+s'}]_{r'}\]
such that $g_j(z,w)$ is a homogeneous polynomial of degree one in $(z,w)\in \mathbb C^{r+s}$ for $1$ $\le$ $j$ $\le$ $r'+s'$.
By Gao-Ng \cite[Propositions 3.1 and 3.2]{Gao_Ng:2018} we have
\[ \sum_{j=1}^{r'}|g_j|^2 - \sum_{l=1}^{s'} |g_{r'+l}|^2
= C \left(\sum_{j=1}^r|z_j|^2 - \sum_{l=1}^s |w_l|^2 \right) \]
for $[z_1,\ldots,z_r,w_1,\ldots,w_s]$ in the domain of $g$, where $C$ is a positive real constant.
We may assume $C=1$ without loss of generality.
Then, by Gao-Ng \cite[Lemma 3.3]{Gao_Ng:2018} we have $r\le r'$, $s\le s'$ and
\[ (g_1,\ldots,g_{r'+s'}) = (z_1,\ldots,z_r,w_1,\ldots,w_s) W^T \]
up to an automorphism of $D_{r,s}$, for some $W\in M(r'+s',r+s;\mathbb C)$ such that $\overline{W}^T {\bf I}_{r',s'} W= {\bf I}_{r,s}$, where ${\bf I}_{m,n}:=\begin{bmatrix} {\bf I}_m & {\bf 0}\\ {\bf 0} & -{\bf I}_n \end{bmatrix} \in M(m+n,m+n;\mathbb C)$ for positive integers $m$ and $n$.
Write $W^T = \begin{bmatrix} W'\\W'' \end{bmatrix}$ for $W'\in M(r,r'+s';\mathbb C)$ and $W''\in M(s,r'+s';\mathbb C)$.
By the arguments of Baouendi-Huang \cite[pp.\,385--386]{Baouendi_Huang:2005} and the fact that $\overline{W}^T {\bf I}_{r',s'} W= {\bf I}_{r,s}$, there exist matrices $V'\in M(r'-r,r'+s';\mathbb C)$ and $V''\in M(s'-s,r'+s';\mathbb C)$ such that the square matrix
\[ M:=\begin{bmatrix}
W'\\ V'\\ W'' \\ V''
\end{bmatrix} \in M(r'+s',r'+s';\mathbb C) \]
is invertible and satisfies $M {\bf I}_{r',s'} \overline{M}^T = {\bf I}_{r',s'}$.
In particular, the map 
\[ D_{r',s'}\ni [\xi_1,\ldots,\xi_{r'+s'}] \mapsto [(\xi_1,\ldots,\xi_{r'+s'}) M] \in D_{r',s'} \]
is an automorphism of $D_{r',s'}$.
Moreover, up to an automorphism of $D_{r,s}$ we have
\[ \begin{split}
(g_1,\ldots,g_{r'+s'}) 
=& (z_1,\ldots,z_r,w_1,\ldots,w_s) W^T\\
=&(z_1,\ldots,z_r,{\bf 0};w_1,\ldots,w_s,{\bf 0})_{r'} M, 
\end{split}\]
where any $(z_1,\ldots,z_{r'+s'})\in \mathbb C^{r'+s'}$ is written as 
\[ (z_1,\ldots,z_{r'+s'})=(z_1,\ldots,z_{r'};z_{r'+1},\ldots,z_{r'+s'})_{r'}.\]  
\end{proof}

\section{Fibral-image-preserving maps and moduli maps}
Let $F:D^{\mathrm{I}}_{q,p}\to D^{\mathrm{I}}_{q',p'}$ be a proper holomorphic map, where $p\ge q\ge 2$ and $q'<p$.
By Proposition \ref{Pro:FIP1}, there is a local moduli map $g:U\subset D_{q,p} \to D_{q',p'}$ of $F$, where $U\subset D_{q,p}$ is a connected open subset.
Write $g([A,B]_q):=[g_1([A,B]_q),g_2([A,B]_q)]_{q'}$.
For any $[A,B]_q \in U \subset D_{q,p}$ and $Z\in [A,B]_q^\sharp$, we have $AZ=B$.
Therefore, $F([A,AZ]_q^\sharp)$ $\subset$ $g([A,AZ]_q)^\sharp$ so that
\[ g_1([A,AZ]_q) F(Z) = g_2([A,AZ]_q)\]
for all $[A]\in \mathbb P^{q-1}$ and all $Z\in D^{\mathrm{I}}_{q,p}$ such that $[A,AZ]_q\in U$ (cf. Seo \cite[p.\,443]{Seo:2015}).
Then, we have the following basic fact.
\begin{lemma}\label{lem:Ch_di_type}
Let $F$ be as defined in the above.
Suppose $F$ has a moduli map $g$ $:$ $D_{q,p}$ $\to$ $D_{q',p'}$ given by $g([A,B]_q)=[A,{\bf 0};B,{\bf 0}]_{q'}$ so that $q'\ge q \ge 2$ and $p'\ge p\ge 2$, where $A=(a_1,\ldots,a_q)$ and $B=(b_1,\ldots,b_p)$ with $[a_1,\ldots,a_q,b_1,\ldots,b_p]\in D_{q,p} \subset \mathbb P^{q+p-1}$.
Then, we have
\[ F(Z) =  \begin{bmatrix}
Z & {\bf 0} \\
{\bf 0} & F_{22}(Z)
\end{bmatrix} \]
for some holomorphic map $F_{22}:D^{\mathrm{I}}_{q,p} \to 
D^{\mathrm{I}}_{q'-q,p'-p}$ when $p'>p$ and $q'>q$, i.e., $F$ is of diagonal type.

In other words, when $p'>p\ge 2$ and $q'>q\ge 2$, the map $g:D_{q,p}\to D_{q',p'}$ given by $g([A,B]_q)=[A,{\bf 0};B,{\bf 0}]_{q'}$ is a moduli map of all holomorphic maps $G_h:D^{\mathrm{I}}_{q,p}\to D^{\mathrm{I}}_{q',p'}$ defined by $G_h(Z):=\begin{bmatrix}
Z & {\bf 0}\\ {\bf 0} & h(Z)
\end{bmatrix}$, where $h:D^{\mathrm{I}}_{q,p} \to D^{\mathrm{I}}_{q'-q,p'-p}$ are holomorphic maps.
\end{lemma}
\begin{proof}
From the assumption, $g$ is actually a linear embedding from $\mathbb P^{q+p-1}$ to $\mathbb P^{q'+p'-1}$ which maps $D_{q,p}$ into $D_{q',p'}$.
Writing 
\[ F(Z) = \begin{bmatrix}
F_{11}(Z) & F_{12}(Z)\\
F_{21}(Z) & F_{22}(Z)
\end{bmatrix}, \]
it follows from $g_1([A,AZ]_q) F(Z) = g_2([A,AZ]_q)$ that
\[ A F_{11}(Z) = AZ,\quad AF_{12}(Z) = {\bf 0} \]
for all $[A]\in \mathbb P^{q-1}$ and all $Z\in D^{\mathrm{I}}_{q,p}$.
We can take $[A_j]\in \mathbb P^{q-1}$, $1\le j\le q$, so that $A_1,\ldots,A_q$ are $\mathbb C$-linearly independent.
Letting $M$ be the $q$-by-$q$ matrix with row vectors $A_1,\ldots,A_q$, $M$ is invertible, $M F_{11}(Z) = MZ$ and $M F_{12}(Z)={\bf 0}$ so that $F_{11}(Z)=Z$ and $F_{12}(Z)\equiv {\bf 0}$.
This implies that
\[ F(Z) =  \begin{bmatrix}
Z & {\bf 0} \\
F_{21}(Z) & F_{22}(Z)
\end{bmatrix}. \] 
We can then deduce from the maximum principle and the homogeneity of the domains that $F_{21}(Z)\equiv {\bf 0}$ (cf. \cite{Ng:2015,Seo:2015}).
Hence, we have 
\[ F(Z) =  \begin{bmatrix}
Z & {\bf 0} \\
{\bf 0} & F_{22}(Z)
\end{bmatrix}. \] 
Since $F:D^{\mathrm{I}}_{q,p}\to D^{\mathrm{I}}_{q',p'}$ is a proper holomorphic map, the matrix
\[ {\bf I}_{p'}-\overline{F(Z)}^T F(Z) = \begin{bmatrix}
{\bf I}_{p} - \overline{Z}^T Z & {\bf 0}\\
{\bf 0} & {\bf I}_{p'-p} - \overline{F_{22}(Z)}^T F_{22}(Z)
\end{bmatrix} \]
is positive definite and so is ${\bf I}_{p'-p} - \overline{F_{22}(Z)}^T F_{22}(Z)$ when $p'>p$ and $q'>q$.
Thus, $F_{22}$ is a holomorphic map from $D^{\mathrm{I}}_{q,p}$ to $D^{\mathrm{I}}_{q'-q,p'-p}$.

Now, we have $g_1([A,AZ]_q)=[A,{\bf 0}]$ and $g_2([A,AZ]_q)=[AZ,{\bf 0}]$ so that
\[ g_1([A,AZ]_q) G_h(Z)
=[A,{\bf 0}] \begin{bmatrix}
Z & {\bf 0}\\ {\bf 0} & h(Z)
\end{bmatrix}
=[AZ,{\bf 0}]
= g_2([A,AZ]_q) \]
for any $[A]\in \mathbb P^{q-1}$ and $Z\in D^{\mathrm{I}}_{q,p}$.
Recall that for any $Z\in D^{\mathrm{I}}_{q,p}$ we have
\[ Z^\sharp = \{[A,AZ]_q\in D_{q,p}: [A]\in \mathbb P^{q-1}\} \]
and
\[ (G_h(Z))^\sharp
= \{[A',A' G_h(Z)]_{q'}\in D_{q',p'}: [A']\in \mathbb P^{q'-1} \}. \]
In other words, $g(Z^\sharp) \subset (G_h(Z))^\sharp$ for any $Z\in D^{\mathrm{I}}_{q,p}$, i.e., each $G_h$ is a moduli map of $g$.
Thus, $g$ is a moduli map of $G_h$'s by Ng \cite[Proposition 2.3]{Ng:2015}.  
\end{proof}

We now state an important result of Ng-Tu-Yin \cite{NTY:2016} that is a key tool in our study.

\begin{proposition}[cf.\,Ng-Tu-Yin $\text{\cite[Proposition 3.9]{NTY:2016}}$]\label{Pro:Pro3.9_NTY}
Let $f:D^{\mathrm{I}}_{p,q}\to D^{\mathrm{I}}_{p',q'}$ be a proper holomorphic map that maps every $(p-1,q)$-subspace of $D^{\mathrm{I}}_{p,q}$ into a $(p'-1,q')$-subspace of $D^{\mathrm{I}}_{p',q'}$, where $p,q\ge 2$.
Suppose $f$ maps a general $(p-1,q)$-subspace of $D^{\mathrm{I}}_{p,q}$ into a unique $(p'-1,q')$-subspace of $D^{\mathrm{I}}_{p',q'}$.
Then, $f$ induces a rational map $\widetilde g:D_{p,q}\dashrightarrow D_{p',q'}$ such that $f([A,B]_p^\sharp)\subset \widetilde g([A,B]_p)^\sharp$ for a general point $[A,B]_p\in D_{p,q}$.
If in addition that $q\ge p' \ge 3$, then $\widetilde g$ is a rational proper map.
\end{proposition}
\begin{proof}
Define
\[ W:=\left\{[A,B]_p\in D_{p,q}: \begin{split}&f([A,B]_p^\sharp)\text{ is contained in a unique}\\ &\text{$(p'-1,q')$-subspace of $D^{\mathrm{I}}_{p',q'}$}\end{split}\right\}, \]
which is a dense open subset of $D_{p,q}$ such that $W=D_{p,q}\smallsetminus V$ for some proper complex-analytic subvariety $V\subset D_{p,q}$ from the assumption. In particular, $W\subset D_{p,q}$ is open.
From the settings, $f:D^{\mathrm{I}}_{p,q}\to D^{\mathrm{I}}_{p',q'}$ is a fibral-image-preserving map such that for any $[A,B]_p\in W$, $f([A,B]_p^\sharp)$ is contained in a unique fibral image in $D^{\mathrm{I}}_{p',q'}$.
Then, $f$ has local (holomorphic) moduli maps which are defined around every point $x\in W\subset D_{p,q}$ by \cite[Proposition 2.15]{Ng:2015}.
More precisely, for any point $x_0\in W$, there is an open neighborhood $U_{x_0}$ of $x_0$ in $D_{p,q}$ and a local moduli map $g_{x_0}:U_{x_0}\subset D_{p,q} \to D_{p',q'}$ of $f$.
If $g_j:U_j\subset D_{p,q}\to D_{p',q'}$, $j=1,2$, are local moduli maps for $f$ such that $U_j \subseteq W$ is a connected open subset, $j=1,2$, with $U_1\cap U_2\neq \varnothing$, then for any $x\in U_1\cap U_2\subset W$, $f(x^\sharp) \subset g_j(x)^\sharp$ for $j=1,2$. Since $f(x^\sharp)$ is contained in a unique $(p'-1,q')$-subspace for all $x\in W$, we have
\[ g_1|_{U_1\cap U_2} \equiv g_2|_{U_1\cap U_2} \]
(cf.\,\cite[Corollary 2.5]{Ng:2015}).
This shows that we have a local (holomorphic) moduli map $g$ $:$ $U_1\cup U_2$ $\subset$ $D_{p,q}$ $\to$ $D_{p',q'}$ for $f$ such that $g|_{U_j}=g_j$, $j=1,2$.
Therefore, putting all these local moduli maps together, we have a holomorphic map $\widetilde g:W = D_{p,q}\smallsetminus V \to D_{p',q'}$ such that $\widetilde g$ is a moduli map for $f$, i.e.,
\[ f([A,B]_p^\sharp) \subseteq \widetilde g([A,B]_p)^\sharp \]
for all $[A,B]_p\in W$.
As in Ng \cite[Proof of Theorem 1.3]{Ng:2015}, the map $\widetilde g:W\to D_{p',q'}$ actually extends to a meromorphic map from $D_{p,q}$ to $D_{p',q'}$, still denoted by $\widetilde g$, and it extends to a rational map $\hat g$ from $\mathbb P^{p+q-1}$ to $\mathbb P^{p'+q'-1}$ by using Hartogs-type extension theorem.

It remains to show that $\widetilde g$ is proper for the first assertion under the assumption that $q>p'\ge 3$.
Denote by $I$ the set of indeterminacy of $\hat g$.
Assume the contrary that $\widetilde g$ is not proper, i.e., there exists $x_0\in \partial D_{p,q} \smallsetminus I$ such that $\hat g(x_0)\in D_{p',q'}$.
In what follows we will make use of the double fibration (\ref{Eq:HDF2})
\[ \mathbb P^{p+q-1} \xleftarrow{\hat\pi^1_{p,q}} \mathcal F^{1,p}_{p+q} \xrightarrow{\hat\pi^2_{p,q}} G(p,q) \]
as in \cite[Proof of Theorem 1.3]{Ng:2015}.
For any $x\in G(p,q)$ we also write $x^\sharp:=\hat\pi^1_{p,q}((\hat\pi^2_{p,q})^{-1}(x))$.
Similarly, for any $y\in \mathbb P^{p+q-1}$ we write $y^\sharp:=\hat\pi^2_{p,q}((\hat\pi^1_{p,q})^{-1}(y))$.
Then, the arguments in \cite[Proof of Theorem 1.3]{Ng:2015} show that there is a one-parameter family of $(p-1,q-1)$-subspaces $\lambda(t) \subset G(p,q)$, $1-\varepsilon\le t \le 1$, such that $\lambda(1)\cap \partial D^{\mathrm{I}}_{p,q}=x_0^\sharp\cap \partial D^{\mathrm{I}}_{p,q}\subset \partial D^{\mathrm{I}}_{p,q}$, and $f$ maps the $(p-1,q-1)$-subspaces $\lambda(t)\cap D^{\mathrm{I}}_{p,q}$, $1-\varepsilon\le t <1$, into some $(p'-2,q'-1)$-subspaces of $D^{\mathrm{I}}_{p',q'}$, where $\varepsilon>0$.
Actually, by the same arguments, for any one-parameter family of $(p-1,q-1)$-subspaces $X_{p-1,q-1}(t) \subset G(p,q)$, $1-\varepsilon\le t\le 1$, such that $X_{p-1,q-1}(1)\cap \partial D^{\mathrm{I}}_{p,q}=x_0^\sharp\cap \partial D^{\mathrm{I}}_{p,q}\subset \partial D^{\mathrm{I}}_{p,q}$ and $X_{p-1,q-1}(t) \subset (\gamma(t))^\sharp$ for all $t$, we can deduce that $f(X_{p-1,q-1}(t)\cap D^{\mathrm{I}}_{p,q})$ lies in some $(p'-2,q'-1)$-subspaces of $D^{\mathrm{I}}_{p',q'}$, where $\gamma(t)\subset D_{p,q} \smallsetminus I$ is a real curve.

Under the assumptions of Proposition \ref{Pro:Pro3.9_NTY}, we have
\begin{lemma}\label{lem:sp_to_sp1}
Suppose $q\ge p'\ge 3$.
Let $X_{p-1,q}$ be a $(p-1,q)$-subspace of $D^{\mathrm{I}}_{p,q}$ such that $f(X_{p-1,q})$ lies in a unique $(p'-1,q')$-subspace.
Then, $f$ cannot map every $(p-1,q-1)$-subspace contained in $X_{p-1,q}$ into a $(p'-2,q'-1)$-subspace.
\end{lemma}

Suppose the above lemma does not hold so that $f$ maps every $(p-1,q-1)$-subspace contained in $X_{p-1,q}$ into a $(p'-2,q'-1)$-subspace.
Then, $f$ induces a holomorphic map $\widetilde f:D^{\mathrm{I}}_{q,p-1}\to D^{\mathrm{I}}_{p'-1,q'}$ which maps every $(q-1,p-1)$-subspace of $D^{\mathrm{I}}_{q,p-1}$ into a $(p'-2,q')$-subspace $D^{\mathrm{I}}_{p'-1,q'}$.
More precisely, $\widetilde f$ is defined so that $\nu\circ \widetilde f(Z):= f((\iota(Z))^T)$ for all $Z\in D^{\mathrm{I}}_{q,p-1}$, where
$\nu:D^{\mathrm{I}}_{p'-1,q'}\hookrightarrow D^{\mathrm{I}}_{p',q'}$ is some standard embedding and $\iota:D^{\mathrm{I}}_{q,p-1}\hookrightarrow D^{\mathrm{I}}_{q,p}$ is the standard embedding such that
the map $H:D^{\mathrm{I}}_{p-1,q}\to D^{\mathrm{I}}_{p,q}$ defined by $H(W)$ $:=$ $(\iota(W^T))^T$, $W\in D^{\mathrm{I}}_{p-1,q}$, has the image $H(D^{\mathrm{I}}_{p-1,q})=X_{p-1,q}$.
Since $q\ge p' >p'-1\ge 2$, by \cite[Theorem 1.2]{Ng:2015}, $\widetilde f(D^{\mathrm{I}}_{q,p-1})$ lies in a single $(p'-2,q')$-subspace.
But this implies that $f(X_{p-1,q})$ lies in more than one $(p'-1,q')$-subspace, a plain contradiction. The proof of the lemma is complete.

We can take any real curve $\gamma(t) \subset D_{p,q}\smallsetminus I$, $1-\varepsilon\le t\le 1$, such that $f$ maps each of the $(p-1,q)$-subspaces $(\gamma(t))^\sharp \cap D^{\mathrm{I}}_{p,q}$, $1-\varepsilon\le t<1$, into a unique $(p'-1,q')$-subspace.
In addition, we can find a one-parameter family $\lambda(t)$ of $(p-1,q-1)$-subspaces in $G(p,q)$ with $\lambda(t)\subset (\gamma(t))^\sharp$ for all $t$, and $\lambda(1)\cap \partial D^{\mathrm{I}}_{p,q} = x_0^\sharp \cap \partial D^{\mathrm{I}}_{p,q}=(\gamma(1))^\sharp \cap \partial D^{\mathrm{I}}_{p,q}$. However, by the assumption $\hat g(x_0) \in D_{p',q'}$, the above arguments taken from \cite[Proof of Theorem 1.3]{Ng:2015} would imply that $f$ maps every $(p-1,q-1)$-subspace of $(\gamma(t))^\sharp \cap D^{\mathrm{I}}_{p,q}$ into a $(p'-2,q'-1)$-subspace, which contradicts with the assertion of Lemma \ref{lem:sp_to_sp1}.
Hence, we have $\hat g(x_0)\in \partial D_{p',q'}$ for all the boundary points $x_0\in \partial D_{p,q}\smallsetminus I$, i.e., $\widetilde g$ is a rational proper map.  
\end{proof}
\begin{remark}\label{remark:6.3}
Actually, if the proper holomorphic map $f:D^{\mathrm{I}}_{p,q}\to D^{\mathrm{I}}_{p',q'}$ {\rm($p,q\ge 2$)} maps every $(p-1,q)$-subspace of $D^{\mathrm{I}}_{p,q}$ to a unique $(p'-1,q')$-subspace of $D^{\mathrm{I}}_{p',q'}$ and $q\ge p'\ge 3$, then $f$ induces a proper holomorphic map $\widetilde g:D_{p,q}\to D_{p',q'}$ such that $f([A,B]_p^\sharp)\subset \widetilde g([A,B]_p)^\sharp$ for all points $[A,B]_p\in D_{p,q}$.
\end{remark}

In \cite{Seo:2015}, Seo also pointed out the following fact regarding equivalence classes of fibral-image-preserving maps and the corresponding moduli maps.

\begin{lemma}[$\text{cf.\,Ng \cite[p.\,25]{Ng:2015}, Seo \cite[Remark 4.3, p.\,443]{Seo:2015}}$]
\label{lem:Seo15_Remark4.3}
Let $f$ $:$ $D^{\mathrm{I}}_{q,p}$ $\to$ $D^{\mathrm{I}}_{q',p'}$ be a holomorphic map and $g:D_{q,p} \to D_{q',p'}$ be a meromorphic map such that $f([A,B]_q^\sharp)$ $\subset$ $g([A,B]_q)^\sharp$ for a general point $[A,B]_q\in D_{q,p}$.
If $g$ is equivalent to a meromorphic map $\widetilde g:D_{q,p}\to D_{q',p'}$, then $f$ is equivalent to a map $\widetilde f:D^{\mathrm{I}}_{q,p}\to D^{\mathrm{I}}_{q',p'}$ such that $\widetilde g$ is a local moduli map of $\widetilde f$.

More precisely, we let $\widetilde g:=\Phi\circ g\circ \phi$, where $\phi\in \mathrm{Aut}(D_{q,p})$ and $\Phi\in \mathrm{Aut}(D_{q',p'})$.
Then, there exist $\psi\in \mathrm{Aut}(D^{\mathrm{I}}_{q,p})$ and $\Psi\in \mathrm{Aut}(D^{\mathrm{I}}_{q',p'})$ such that $\widetilde f:=\Psi\circ f\circ \psi$ satisfies
$\widetilde f([A,B]_q^\sharp) \subset \widetilde g([A,B]_q)^\sharp$
for a general point $[A,B]_q\in D_{q,p}$.
\end{lemma}

\subsection{Fibral-image-preserving maps between generalized complex balls}
In this section, we study fibral-image-preserving holomorphic maps between generalized complex balls with respect to the double fibrations constructed in Section \ref{Sec:DF_FIPM}.
We first obtain the following fundamental result for such maps.
\begin{proposition}\label{Pro:GenBall_FIPM1}
Let $q,p,q'$ and $p'$ be positive integers such that $q,q'\ge 2$.
Let $g$ $:$ $U$ $\subset$ $D_{q,p}$ $\to$ $ D_{q',p'}$ be a fibral-image-preserving holomorphic map with respect to the double fibrations
\begin{equation}\label{Eq:DF_SC_1_gen}
D_{q,p} \xleftarrow{\pi_{q,p}^1} \mathbb P^{q-1}\times D^{\mathrm{I}}_{q,p}\xrightarrow{\pi_{q,p}^2} D^{\mathrm{I}}_{q,p},
\end{equation}
\begin{equation}\label{Eq:DF_SC_2_gen}
D_{q',p'} \xleftarrow{\pi_{q',p'}^1} \mathbb P^{q'-1}\times D^{\mathrm{I}}_{q',p'} \xrightarrow{\pi_{q',p'}^2} D^{\mathrm{I}}_{q',p'},
\end{equation}
i.e., for any $Z\in D^{\mathrm{I}}_{q,p}$ such that $Z^\sharp\cap U\neq \varnothing$, we have $g(Z^\sharp\cap U)\subset W^\sharp$ for some $W\in D^{\mathrm{I}}_{q',p'}$.
Suppose there exists $Z_0\in D^{\mathrm{I}}_{q,p}$ such that $Z_0^\sharp\cap U\neq \varnothing$ and $g(Z_0^\sharp\cap U)\subset W_0^\sharp$ for a unique $W_0\in D^{\mathrm{I}}_{q',p'}$.
Then, $g$ has a local moduli map $F:U_{Z_0}\subset D^{\mathrm{I}}_{q,p} \to D^{\mathrm{I}}_{q',p'}$ for some open neighborhood $U_{Z_0}$ of $Z_0$ in $D^{\mathrm{I}}_{q,p}$, and for all $Z\in U_{Z_0}$, $g(Z^\sharp\cap U)$ lies inside a unique fibral image $F(Z)^\sharp$ in $D_{q',p'}$.
\end{proposition}
\begin{proof}
Write $g=[g_1;g_2]_{q'}$.
Let $L_{g}:\mathbb P^{q-1}\times D^{\mathrm{I}}_{q,p}\times D^{\mathrm{I}}_{q',p'}\to M(1,p';\mathbb C)\cong \mathbb C^{p'}$ be the map defined by
\[ L_g([A],Z,W):=g_1([A,AZ])W-g_2([A,AZ]) \]
for $([A],Z,W)\in \mathbb P^{q-1}\times D^{\mathrm{I}}_{q,p}\times D^{\mathrm{I}}_{q',p'}$.

Define $L_{g,[A],Z}:D^{\mathrm{I}}_{q',p'}\to M(1,p';\mathbb C)\cong \mathbb C^{p'}$ by
\[ L_{g,[A],Z}(W):=L_g([A],Z,W)\quad \forall\;W\in D^{\mathrm{I}}_{q',p'}. \]
Then, $L_{g,[A],Z}$ is a linear map for all $([A],Z)\in \mathbb P^{q-1}\times D^{\mathrm{I}}_{q,p}$.
For any $Z\in D^{\mathrm{I}}_{q,p}$ we define
\[ V_Z:=\bigcap_{[A]\in \mathbb P^{q-1}} \mathrm{Zero} (L_{g,[A],Z}). \]
By the assumption, $V_{Z_0}=\{W_0\}$.
There exist a positive integer $s$ and distinct $[A_j]\in \mathbb P^{q-1}$, $1\le j\le s$, such that
\[ V_{Z_0}=\bigcap_{[A]\in \mathbb P^{q-1}} \mathrm{Zero} (L_{g,[A],Z})
=\bigcap_{j=1}^s \mathrm{Zero} (L_{g,[A_j],Z_0}). \]
Define $\widetilde V_{Z}:=\bigcap_{j=1}^s \mathrm{Zero} (L_{g,[A_j],Z})$ for every $Z\in D^{\mathrm{I}}_{q,p}$.
Then, $V_{Z} \subseteq \widetilde V_{Z}$ for all $Z\in D^{\mathrm{I}}_{q,p}$.
Since $\dim_{\mathbb C} \widetilde V_{Z_0} = 0$, there exists an open neighborhood $U_{Z_0}$ of $Z_0$ in $D^{\mathrm{I}}_{q,p}$ such that
\[ \dim_{\mathbb C} V_Z \le  \dim_{\mathbb C} \widetilde V_Z
\le \dim_{\mathbb C} \widetilde V_{Z_0} = 0 \]
for all $Z\in U_{Z_0}$ by the upper semi-continuity of the function $h(Z):=\dim_{\mathbb C} \widetilde V_Z$.
Since $V_Z\neq \varnothing$ for all $Z\in D^{\mathrm{I}}_{q,p}$, we have $\dim_{\mathbb C} V_Z=0$ for all $Z\in U_{Z_0}$.
In particular, since $V_Z$ is a linear section of $D^{\mathrm{I}}_{q',p'}$, $V_Z=\{W_{Z}\}$ for all $Z\in U_{Z_0}$.
Define $F:U_{Z_0}\to D^{\mathrm{I}}_{q',p'}$ by $F(Z):=W_Z$, we have the desired local moduli map $F$ of $g$, i.e.,
$g(Z^\sharp \cap U) \subset F(Z)^\sharp$
for $Z\in U_{Z_0}$.  
\end{proof}

We may also extend the definition of the double fibrations to topological closures of the type-$\mathrm{I}$ irreducible bounded symmetric domains in complex Euclidean spaces and topological closures of the generalized complex balls in complex projective spaces (cf.\,Seo \cite[Section 4.1]{Seo:2015}), namely, we have the double fibrations
\[ \overline{D_{q,p}} \xleftarrow{\widetilde\pi_{q,p}^1} \mathbb P^{q-1}\times \overline{D^{\mathrm{I}}_{q,p}}\xrightarrow{\widetilde\pi_{q,p}^2} \overline{D^{\mathrm{I}}_{q,p}} \]
for integers $p,q\ge 2$.
Then, for any $W_1,W_2\in \overline{D^{\mathrm{I}}_{q,p}}$, $W_1^\sharp=W_2^\sharp$ if and only if $W_1=W_2$ by Ng \cite[Proposition 3.2]{Ng:2015}.

We now restrict to the special case where $q'=q+1$ and $3\le q\le p-2$, and we obtain the following rigidity theorem for fibral-image-preserving rational proper maps between the generalized complex balls $D_{q,p}$ and $D_{q+1,p'}$ with $3\le q\le p-2$ and $p\le p'$.
\begin{proposition}\label{Pro:RatProperM_FIP1}
Let $g:U\subset D_{q,p}\to D_{q+1,p'}$ be a rational proper map that is fibral-image-preserving with respect to the double fibrations
\begin{equation}\label{Eq:DF_SC_1}
D_{q,p} \xleftarrow{\pi_{q,p}^1} \mathbb P^{q-1}\times D^{\mathrm{I}}_{q,p}\xrightarrow{\pi_{q,p}^2} D^{\mathrm{I}}_{q,p},
\end{equation}
\begin{equation}\label{Eq:DF_SC_2}
D_{q+1,p'} \xleftarrow{\pi_{q+1,p'}^1} \mathbb P^{q}\times D^{\mathrm{I}}_{q+1,p'} \xrightarrow{\pi_{q+1,p'}^2} D^{\mathrm{I}}_{q+1,p'},
\end{equation}
where $3\le q\le p-2$ and $p'\ge p$, i.e.,
for any $Z\in D^{\mathrm{I}}_{q,p}$ such that $Z^\sharp\cap U\neq \varnothing$, we have $g(Z^\sharp\cap U)\subset W^\sharp$ for some $W\in D^{\mathrm{I}}_{q+1,p'}$.
Then, $g$ is linear, i.e., $\deg(g)=1$, and $g$ is equivalent to the map
\[ [z_1,\ldots,z_q;z_{q+1},\ldots,z_{q+p}]_q\mapsto
[z_1,\ldots,z_q,0;z_{q+1},\ldots,z_{q+p},{\bf 0}]_{q+1}. \]
\end{proposition}
\begin{proof}
If $p'=p$, then the result follows from Theorem \ref{thm:thm1.1&1.4_BH05}.
Thus, we may suppose $p'>p$ from now on.
Since $g$ is a rational proper map, by Theorem \ref{thm:BEH11_MainThm} we have the following two possibilities.
\begin{enumerate}
\item[(1)] $g$ is equivalent to the map $\widetilde g:\widetilde U \subset D_{q,p}\to D_{q+1,p'}$ given by
\[ \widetilde g([A,B]_q)
=[A,\phi(A,B);B,\phi(A,B),{\bf 0}]_{q+1} \]
for all $[A,B]_q \in U \subset D_{q,p}$, where $\phi(z)={P(z)\over Q(z)}$ is a rational function such that $P(z),Q(z)\in \mathbb C[z]$ are homogeneous polynomials with $\deg P = \deg Q+1$ whenever $\deg P\ge 1$, $z\in \mathbb C^{p+q}$.
\item[(2)] $g$ is of the form $[h,\phi_1,\ldots,\phi_q;\psi_1,\ldots,\psi_{p'-1},h]_{q+1}$, and this induces a rational proper map $\nu:=[\phi_1,\ldots,\phi_q;\psi_1,\ldots,\psi_{p'-1}]_{q}:U'\subset D_{q,p} \to D_{q,p'-1}$.
By Theorem \ref{thm:thm1.1&1.4_BH05}, $\nu$ is linear and
\[ \nu\circ \sigma([A,B]_q)  \cdot M = [A;B,{\bf 0}]_{q} \]
for some $\sigma\in {\rm Aut}(D_{q,p})$ and some $M\in GL(q+p'-1,\mathbb C)$ such that
\[ M \begin{bmatrix}
{\bf I}_{q} & {\bf 0}\\
{\bf 0} & -{\bf I}_{p'-1}
\end{bmatrix} \overline{M}^T = \begin{bmatrix}
{\bf I}_{q} & {\bf 0}\\
{\bf 0} & -{\bf I}_{p'-1}
\end{bmatrix} \]
(see \cite[pp.\,379--380]{Baouendi_Huang:2005}).
Then, $g$ is equivalent to the map $\widetilde g = \widetilde \tau\circ g\circ \sigma:\widetilde U \subset D_{q,p}\to D_{q+1,p'}$ given by
\[ \widetilde g([A,B]_q):=[\widetilde h(A,B),A;B,{\bf 0},\widetilde h(A,B)]_{q+1}, \]
where $\widetilde \tau \in {\rm Aut}(D_{q+1,p'})$ is defined by
\[ \widetilde \tau([C;D]_{q+1}):=[C;D]_{q+1} \cdot \begin{bmatrix}
1 & & \\
 & M & \\
 & & 1
\end{bmatrix}. \]
\end{enumerate}
Combining these two cases, we may replace $g$ by the map $\widetilde g:\widetilde U \subset D_{q,p}\to D_{q+1,p'}$ given by
\[ \widetilde g([A,B]_q)
=[A,\phi(A,B);B,\phi(A,B),{\bf 0}]_{q+1} \]
in the consideration. (Noting that $\widetilde g$ is holomorphic on $\widetilde U$.)

We claim that $\widetilde g$ maps every fibral image of the double fibration {\rm(\ref{Eq:DF_SC_1})} to at least two fibral images of the double fibration
\[ \overline{D_{q+1,p'}} \xleftarrow{\widetilde\pi_{q+1,p'}^1} \mathbb P^{q}\times \overline{D^{\mathrm{I}}_{q+1,p'}} \xrightarrow{\widetilde\pi_{q+1,p'}^2} \overline{D^{\mathrm{I}}_{q+1,p'}}, \]
i.e., for any $Z\in D^{\mathrm{I}}_{q,p}$ such that $Z^\sharp\cap \widetilde U\neq \varnothing$, $\widetilde g(Z^\sharp\cap \widetilde U)\subset W_1^\sharp\cap W_2^\sharp$ for some distinct $W_1,W_2\in \overline{D^{\mathrm{I}}_{q+1,p'}}$.

From the assumption, for any $Z\in D^{\mathrm{I}}_{q,p}$ such that $Z^\sharp \cap \widetilde U\neq \varnothing$, we have $\widetilde g(Z^\sharp\cap \widetilde U)\subset W^\sharp$ for some $W\in D^{\mathrm{I}}_{q+1,p'}$.
On the other hand, from the expression of $\widetilde g$, we also have
\[ \widetilde g(Z^\sharp\cap \widetilde U) \subset
\begin{bmatrix}
Z &{\bf 0}& {\bf 0}\\
{\bf 0} & 1 & {\bf 0}
\end{bmatrix}^\sharp. \]
Then, for each $Z\in D^{\mathrm{I}}_{q,p}$ such that $Z^\sharp\cap \widetilde U\neq \varnothing$, $\widetilde g(Z^\sharp\cap \widetilde U)$ lies inside two fibral images $\begin{bmatrix}
Z &{\bf 0}& {\bf 0}\\
{\bf 0} & 1 & {\bf 0}
\end{bmatrix}^\sharp$ and $W^\sharp$ for some $W\in D^{\mathrm{I}}_{q+1,p'}$.
In addition, $\begin{bmatrix}
Z &{\bf 0}& {\bf 0}\\
{\bf 0} & 1 & {\bf 0}
\end{bmatrix}^\sharp$ and $W^\sharp$ are distinct because $\begin{bmatrix}
Z &{\bf 0}& {\bf 0}\\
{\bf 0} & 1 & {\bf 0}
\end{bmatrix}\in \partial D^{\mathrm{I}}_{q+1,p'}$ and $W \in D^{\mathrm{I}}_{q+1,p'}$ so that $\begin{bmatrix}
Z &{\bf 0}& {\bf 0}\\
{\bf 0} & 1 & {\bf 0}
\end{bmatrix}\neq W$.
The claim is proved. This actually shows that for any $Z\in D^{\mathrm{I}}_{q,p}$ such that $Z^\sharp\cap \widetilde U\neq \varnothing$, $\widetilde g(Z^\sharp\cap \widetilde U)$ lies inside a $(q-1)$-dimensional projective subspace in $D_{q+1,p'}$.

Note that fibral images in $D_{q,p}$ (resp.\,$D_{q+1,p'}$) are $(q-1)$-dimensional (resp.\,$q$-dimensional) projective subspaces lying inside $D_{q,p}$ (resp.\,$D_{q+1,p'}$).
The above claim shows that $\widetilde g:\widetilde U \to D_{q+1,p'}\subset \mathbb P^{q+p'}$ is a holomorphic map such that for every $(q-1)$-dimensional projective subspace $\Pi$ of $\mathbb P^{q+p-1}$ with $\widetilde U\cap \Pi\neq\varnothing$, we have 
\[ \widetilde g(\widetilde U\cap \Pi)\subset \Pi'\]
for some $(q-1)$-dimensional projective subspace $\Pi'$ in $D_{q+1,p'}$.
By Ng \cite[Proposition 4.2]{Ng:2015}, either (a) $\widetilde g(\widetilde U)$ is contained in a single $(q-1)$-dimensional projective subspace of $\mathbb P^{q+p'}$, or (b) $\deg(\widetilde g)=1$.
Since $\dim_{\mathbb C} \widetilde g(\widetilde U) = q+p-1>q-1$ by $p\ge 2$, Case (a) is actually impossible.
Hence, we have $\deg(\widetilde g)=1$ and thus $\deg(g)=1$, i.e., $g$ is linear.
By Proposition \ref{pro:linear_rational_proper}, $g$ is equivalent to the map
\[ [z_1,\ldots,z_q;z_{q+1},\ldots,z_{q+p}]_q\mapsto
[z_1,\ldots,z_q,0;z_{q+1},\ldots,z_{q+p},{\bf 0}]_{q+1}. \]  
\end{proof}

Before obtaining the general version of Proposition \ref{Pro:RatProperM_FIP1}, we need the following lemma regarding the partitions of matrices in $D^{\rm I}_{q',p'}$.

\begin{lemma}\label{lem:Part_Mat_1}
Let $W\in D^{\rm I}_{q',p'}$, where $q',p'\ge 1$ are integers.
\begin{enumerate}
\item Write $W=[v_1,\ldots,v_m]$, where $v_j\in M(q',k_j;\mathbb C)$, $1\le j\le m$, and $k_j$, $1\le j\le m$, are positive integers such that $\sum_{j=1}^m k_j = p'$.
Then, $v_j \in D^{\rm I}_{q',k_j}$ for $1\le j\le m$.
\item Write $W=[\eta_1^T,\ldots,\eta_n^T]^T$, where $\eta_j\in M(l_j,p';\mathbb C)$, $1\le j\le n$, and $l_j$, $1\le j\le n$, are positive integers such that $\sum_{j=1}^n l_j = q'$.
Then, $\eta_j \in D^{\rm I}_{l_j,p'}$ for $1\le j\le n$.
\end{enumerate}
In particular, writing
\[ W=\begin{bmatrix}
W_{11} &\cdots & W_{1m}\\
\vdots & \ddots & \vdots\\
W_{n1} &\cdots & W_{nm}
\end{bmatrix} \]
as a block matrix such that $W_{ij}\in M(l_i,k_j;\mathbb C)$, where $l_i\ge 1$, and $k_j\ge 1$ for $1\le i\le n$, $1\le j\le m$, with $\sum_{i=1}^n l_i = q'$ and $\sum_{j=1}^m k_j = p'$, we have $W_{ij}\in D^{\rm I}_{l_i,k_j}$.
\end{lemma}
\begin{proof}
We first consider Case (1).
It is well-known that the map from $M(q',p';\mathbb C)$ to $M(p',q';\mathbb C)$ given by $Z\mapsto Z^T$ actually induces a biholomorphism from $D^{\rm I}_{q',p'}$ onto $D^{\rm I}_{p',q'}$.
Therefore, we have $W\in D^{\rm I}_{q',p'}$ if and only if $W\in M(q',p',\mathbb C)$ such that ${\bf I}_{q'} - W\overline{W}^T>0$.
Then, we have
\[ {\bf I}_{q'}-\sum_{j=1}^m v_j \overline{v_j}^T = {\bf I}_{q'}- W\overline{W}^T > 0 \]
and thus for $1\le j\le m$, ${\bf I}_{q'} - v_j\overline{v_j}^T > \sum_{1\le i\le m,\;i\neq j} v_i\overline{v_i}^T$.
For any $\alpha\in M(1,q';\mathbb C)$, $\alpha\neq 0$, we have
$\alpha v_i\overline{v_i}^T \overline{\alpha}^T = (\alpha v_i) \overline{\alpha v_i}^T \ge 0$ for $1\le i\le m$.
Therefore, $v_i \overline{v_i}^T \ge 0$ for $1\le i\le m$ and thus
\[ {\bf I}_{q'} - v_j\overline{v_j}^T > \sum_{1\le i\le m,\;i\neq j} v_i\overline{v_i}^T\ge 0, \]
i.e., ${\bf I}_{q'} - v_j\overline{v_j}^T > 0$, and $v_j\in D^{\rm I}_{q',k_j}$, as desired.

Now, we consider Case (2). We have $W\in M(q',p';\mathbb C)$ such that ${\bf I}_{p'}-\overline{W}^TW > 0$ and
\[ {\bf I}_{p'} -\overline{W}^TW = {\bf I}_{p'} - \sum_{j=1}^n \overline{\eta_j}^T\eta_j. \]
By the same arguments as in Case (1), $\overline{\eta_j}^T\eta_j\ge 0$ for $1\le j\le n$ and thus
\[ {\bf I}_{p'} - \overline{\eta_j}^T\eta_j > \sum_{1\le i\le n,\;i\neq j} \overline{\eta_i}^T\eta_i \ge 0, \]
i.e., ${\bf I}_{p'} - \overline{\eta_j}^T\eta_j > 0$, and $\eta_j\in D^{\rm I}_{l_j,p'}$, as desired.
Finally, the last assertion of this lemma follows readily from both Case (1) and Case (2).
\end{proof}

Now, we are ready to establish the following rigidity theorem for fibral-image-preserving rational proper maps between the generalized complex balls $D_{q,p}$ and $D_{q',p'}$ for $p\ge q\ge 2$, $q'<\min\{2q-1,p\}$, $q\le q'$ and $p\le p'$.

\begin{proposition}\label{pro:rat_proper_FIP_2}
Let $g:U\subset D_{q,p}\to D_{q',p'}$ be a rational proper map that is fibral-image-preserving with respect to the double fibrations
\begin{equation}\label{Eq:DF_SC_3}
D_{q,p} \xleftarrow{\pi_{q,p}^1} \mathbb P^{q-1}\times D^{\mathrm{I}}_{q,p}\xrightarrow{\pi_{q,p}^2} D^{\mathrm{I}}_{q,p},
\end{equation}
\begin{equation}\label{Eq:DF_SC_4}
D_{q',p'} \xleftarrow{\pi_{q',p'}^1} \mathbb P^{q'-1}\times D^{\mathrm{I}}_{q',p'} \xrightarrow{\pi_{q',p'}^2} D^{\mathrm{I}}_{q',p'},
\end{equation}
where $p\ge q\ge 2$, $q'<\min\{2q-1,p\}$, $q\le q'$ and $p\le p'$.
Then, for any $Z\in D^{\mathrm{I}}_{q,p}$ such that $Z^\sharp\cap U\neq \varnothing$, $g(Z^\sharp\cap U)\subset W_1^\sharp\cap W_2^\sharp$ for some distinct $W_1,W_2\in \overline{D^{\mathrm{I}}_{q',p'}}$.
Actually, for every $(q-1)$-dimensional projective subspace $\Pi$ in $D_{q,p}$ with $\Pi\cap U\neq\varnothing$, we have $g(\Pi\cap U)\subset \Pi'$ for some $(q'-2)$-dimensional projective subspace $\Pi'$ in $D_{q',p'}$.
Furthermore, $g$ is linear, i.e., $\deg(g)=1$, and $g$ is equivalent to the map
\[ [z_1,\ldots,z_q;z_{q+1},\ldots,z_{q+p}]_q\mapsto
[z_1,\ldots,z_q,{\bf 0};z_{q+1},\ldots,z_{q+p},{\bf 0}]_{q'} \]
\end{proposition}
\begin{proof}
If $q'=q$ or $p'=p$, then the result follows from Theorem \ref{thm:thm1.1&1.4_BH05}.
Thus, we may suppose $q'>q$ and $p'>p$ from now on.
Since $g$ is a rational proper map, by Theorem \ref{thm:BEH11_MainThm} we have the following two possibilities.
\begin{enumerate}
\item[(1)] $g$ is equivalent to the map $\widetilde g:\widetilde U \subset D_{q,p}\to D_{q',p'}$ given by
\[ \widetilde g([A,B]_q)
=[A,\phi(A,B),{\bf 0};B,\phi(A,B),{\bf 0}]_{q'} \quad \forall\;[A,B]_q \in U \subset D_{q,p},\]
where $\phi=(\phi_1,\ldots,\phi_{\min\{p'-p,q'-q\}})$ such that $\phi_j(z)={p_j(z)\over q_j(z)}$ for some homogeneous polynomials $p_j(z)$ and $q_j(z)$ in $z$ $=$ ($z_1$,$\ldots$,$z_{p+q}$) $\in$ $\mathbb C^{p+q}$ with $\deg p_j=\deg q_j+1$, $1\le j\le \min\{p'-p,q'-q\}$.
\item[(2)] $g$ is of the form
\[ [h,\phi_1,\ldots,\phi_{q'-1};\psi_1,\ldots,\psi_{p'-1},h]_{q'}\]
for some rational functions $h$, $\phi_j$, $1\le j\le q'-1$, and $\psi_i$, $1\le i\le p'-1$, and this induces a rational proper map
\[ \nu:=[\phi_1,\ldots,\phi_{q'-1};\psi_1,\ldots,\psi_{p'-1}]_{q'-1}:U'\subset D_{q,p}\to D_{q'-1,p'-1}. \]
\end{enumerate}
For Case (2), we apply Theorem \ref{thm:BEH11_MainThm} to the rational proper map $\nu$, hence similar to that of $g$, we obtain the two cases for $\nu$ as above, i.e., either (a) $\nu$ is equivalent to the map of the form $[A,\phi^{(1)}(A,B),{\bf 0};B,\phi^{(1)}(A,B),{\bf 0}]_{q'-1}$ or (b) $\nu=[h^{(1)},\nu^{(1)},h^{(1)}]$ induces the rational proper map $\nu^{(1)}$ from $U^{(1)}\subset D_{q,p}$ to $D_{q'-2,p'-2}$ provided that $q'-2\ge q$ and $p'-2\ge p$.
If Case (a) holds, then we are done. Otherwise, $\nu$ falls into Case (b) and we apply Theorem \ref{thm:BEH11_MainThm} to the rational proper map $\nu^{(1)}$.
By induction and by the arguments in the proof of Proposition \ref{Pro:RatProperM_FIP1}, $g$ is eventually equivalent to a rational proper map $\widetilde g:\widetilde U \subset D_{q,p}\to D_{q',p'}$ given by
\begin{equation}\label{Eq:Standard_form}
\widetilde g([A,B]_q)
=[A,\phi(A,B),{\bf 0};B,\phi(A,B),{\bf 0}]_{q'} \quad \forall\;[A,B]_q \in \widetilde U \subset D_{q,p},
\end{equation}
where $\phi=(\phi_1,\ldots,\phi_{\min\{p'-p,q'-q\}})$ and $\phi_j(A,B)$ are some rational functions in $(A,B)$, $1\le j\le \min\{p'-p,q'-q\}$.
That means we can always replace $g$ by such a map $\widetilde g$.
Then, the proof of the first assertion is complete by the same arguments in the proof of Proposition \ref{Pro:RatProperM_FIP1}.
This finishes the proof of the first assertion.

It remains to prove that $g$ is of degree $1$, equivalently, the map $\widetilde g$ in Equation (\ref{Eq:Standard_form}) is of degree $1$.
From the assumption that $q'>q$ and $p'>p$ at the beginning of the proof, we have $\min\{p'-p,q'-q\}\ge 1$.
Since $\widetilde g$ is equivalent to $g$, for any $Z\in D^{\rm I}_{q,p}$ such that $Z^\sharp\cap \widetilde U\neq \varnothing$ we have
\begin{equation}\label{Eq:RPM_FP1}
[A,\phi(A,AZ),{\bf 0}] W = [AZ,\phi(A,AZ),{\bf 0}]
\end{equation} 
for some $W\in D^{\rm I}_{q',p'}$ and for all $[A,AZ]_q\in Z^\sharp\cap \widetilde U$.
We fix such an element $Z$ in the following.
Write $s_0:=\min\{p'-p,q'-q\}$ and write $W=\begin{bmatrix} W_{ij} \end{bmatrix}$ as a block matrix so that $W_{11}\in M(q,p;\mathbb C)$ and $W_{22}\in M(s_0,s_0;\mathbb C)$. Then, from Equation (\ref{Eq:RPM_FP1}) we have
\[ AW_{12}+\phi(A,AZ)W_{22}=\phi(A,AZ). \]
If ${\bf I}_{s_0}-W_{22}$ is invertible, then the above equality yields
\[ \phi(A,AZ) = AW_{12}({\bf I}_{s_0}-W_{22})^{-1} \]
and thus $\phi(A,AZ)$ is of degree $1$ in $A$. In this case, we have
\[ \widetilde g([A,AZ]_q)
=[A,AW_{12}({\bf I}_{s_0}-W_{22})^{-1},{\bf 0};AZ,AW_{12}({\bf I}_{s_0}-W_{22})^{-1},{\bf 0}]_{q'} \]
on the dense open subset $\widetilde U \cap Z^\sharp$ of $Z^\sharp$. Hence, by the identity theorem for meromorphic maps, $\widetilde g$ is well-defined on $Z^\sharp \cong \mathbb P^{q-1}$, i.e., $Z^\sharp \subset \widetilde U$, and the restriction of $\widetilde g$ to $Z^\sharp\cong \mathbb P^{q-1}$ is actually a holomorphic immersion. (Noting that this holds for any $Z\in D^{\rm I}_{q,p}$ such that $Z^\sharp\cap \widetilde U\neq \varnothing$.) Then, one can deduce from the arguments of Ng \cite[Proof of the Main Theorem]{Ng:2013} that $\widetilde g$ is linear, i.e., $\widetilde g$ is of degree $1$.

Thus, to finish the proof, it suffices to prove that ${\bf I}_{s_0}-W_{22}$ is invertible. By Lemma \ref{lem:Part_Mat_1}, we have $W_{22}\in D^{\rm I}_{s_0,s_0}$ so that
\[ {\bf I}_{s_0}-\overline{W_{22}}^T W_{22} > 0.\]
Assume the contrary that ${\bf I}_{s_0}-W_{22}$ is not invertible, equivalently, $\det({\bf I}_{s_0}-W_{22})=0$. Then, $1$ is an eigenvalue of $W_{22}$ and thus there exists an eigenvector $\alpha\in M(s_0,1;\mathbb C)$, $\alpha\neq 0$, of $W_{22}$ such that $W_{22} \alpha = \alpha$. It follows that
\[ \overline{\alpha}^T\left({\bf I}_{s_0}-\overline{W_{22}}^T W_{22}\right) \alpha
= \overline{\alpha}^T \alpha - \overline{W_{22}\alpha}^T W_{22}\alpha
= \overline{\alpha}^T \alpha - \overline{\alpha}^T \alpha
= 0, \]
which contradicts with the fact that ${\bf I}_{s_0}-\overline{W_{22}}^T W_{22}>0$. Hence, ${\bf I}_{s_0}-W_{22}$ is invertible.

As a consequence, $\widetilde g$ must be of degree $1$, and so is $g$. It follows from Proposition \ref{pro:linear_rational_proper} that $g$ is equivalent to the map
\[ [z_1,\ldots,z_q;z_{q+1},\ldots,z_{q+p}]_q\mapsto
[z_1,\ldots,z_q,{\bf 0};z_{q+1},\ldots,z_{q+p},{\bf 0}]_{q'}. \]  
\end{proof}
\begin{remark}
Perhaps Proposition \ref{Pro:RatProperM_FIP1} follows directly from Proposition \ref{pro:rat_proper_FIP_2}, we still keep the statement and the proof of Proposition \ref{Pro:RatProperM_FIP1} because this shows how Proposition \ref{Pro:RatProperM_FIP1} is generalized to Proposition \ref{pro:rat_proper_FIP_2}.
\end{remark}

\subsection{Behaviour of fibral-image-preserving maps between type-$\mathrm{I}$ domains}
\label{subsection:B_FIPM}
In this section, we investigate the behaviour of fibral-image-preserving holomorphic maps from $D^{\mathrm{I}}_{q,p}$ to $D^{\mathrm{I}}_{q',p'}$ with respect to the double fibrations
\[ D_{q,p} \xleftarrow{\pi_{q,p}^1} \mathbb P^{q-1}\times D^{\mathrm{I}}_{q,p}\xrightarrow{\pi_{q,p}^2} D^{\mathrm{I}}_{q,p}, \]
 \[ D_{q',p'} \xleftarrow{\pi_{q',p'}^1} \mathbb P^{q'-1}\times D^{\mathrm{I}}_{q',p'} \xrightarrow{\pi_{q',p'}^2} D^{\mathrm{I}}_{q',p'},\]
where $q,q'\ge 2$.

Let $f:D^{\mathrm{I}}_{q,p}\to D^{\mathrm{I}}_{q',p'}$ be a non-constant holomorphic map which maps every $(q-1,p)$-subspace of $D^{\mathrm{I}}_{q,p}$ into a $(q'-1,p')$-subspace of $D^{\mathrm{I}}_{q',p'}$, where $q,q'\ge 2$.
Then, for any $x\in D_{q,p}$ we define
\[ \mathcal V_{x}:=\{z \in D_{q',p'}: f(x^\sharp) \subset z^\sharp\}=\bigcap_{y\in f(x^\sharp)} y^\sharp \]
so that $\mathcal V_{x}\cong \mathbb P^{q'-n_0(x)-1}$ by Ng \cite[Proposition 3.2]{Ng:2015}, where 
\[ n_0(x)=n_0(f(x^\sharp)):=\min\{ n : f(x^\sharp) \subset D\cong D^{\mathrm{I}}_{n,p'}\subset D^{\mathrm{I}}_{q',p'}\}. \]
Defining
\[ V:=\left\{(x,z)\in D_{q,p}\times D_{q',p'}: z\in \mathcal V_x\right\} , \]
$V$ is a complex-analytic variety.
Let $\phi:=\mathrm{Pr}_1|_V: V\to D_{q,p}$ be the projection map onto the first factor, i.e., $\phi(x,z):=x$ for all $(x,z)\in V$.
Then, for any $x\in D_{q,p}$ we have $\phi^{-1}(x)=\{x\}\times \mathcal V_x \cong \mathcal V_x \cong \mathbb P^{q'-n_0(x)-1}$ and $\phi$ is proper.
For $l\ge 0$ we define
\[ \mathcal U_l:=\left\{ x\in D_{q,p}: \dim_{\mathbb C} \phi^{-1}(x) \ge l \right\} \]
and
\[ \widetilde{\mathcal U_l}:= \left\{(x,z)\in V: \dim_{(x,z)} \phi^{-1}(\phi(x,z)) \ge l \right\} \]
so that $\mathcal U_l=\phi(\widetilde{\mathcal U_l})$.
It is clear that for any $(x,z)\in V$ we have 
\[ \dim_{(x,z)} \phi^{-1}(\phi(x,z))=\dim_{\mathbb C} \phi^{-1}(x) = \dim_{\mathbb C} \mathcal V_x. \]
By Grauert's Semicontinuity Theorem on the fiber dimension, $\widetilde{\mathcal U_l}\subset V$ (resp.\,$\mathcal U_l$ $\subset$  $D_{q,p}$) is a complex-analytic subvariety.
On the other hand, we can rewrite $\mathcal U_l$ as
\[ \mathcal U_l= \left\{ x\in D_{q,p}: f(x^\sharp) \text{ lies in a $(q'-l-1,p')$-subspace of $D^{\mathrm{I}}_{q',p'}$} \right\} \]
for $0\le l\le q'-1$.

Suppose there exists $x_0\in D_{q,p}$ such that $f$ maps $x_0^\sharp$ into a unique $(q'-1,p')$-subspace of $D^{\mathrm{I}}_{q',p'}$.
Then, $\dim_{\mathbb C}\mathcal V_{x_0} = 0$, equivalently, $\dim_{\mathbb C} \phi^{-1}(x_0)=0$.
Hence, $\mathcal U_1 \subset D_{q,p}$ can only be a proper complex-analytic subvariety and thus the set
\[ \left\{ x\in D_{q,p}: \begin{split}
&\text{$f$ maps $x^\sharp$ into a unique}\\
&\text{$(q'-1,p')$-subspace of $D^{\mathrm{I}}_{q',p'}$}
\end{split}\right\} = D_{q,p} \smallsetminus \mathcal U_1 \]
is a dense open subset.

In particular, $f$ must satisfy one of the following two properties.
\begin{enumerate}
\item $f$ maps a general $(q-1,p)$-subspace of $D^{\mathrm{I}}_{q,p}$ into a unique $(q'-1,p')$-subspace of $D^{\mathrm{I}}_{q',p'}$.
\item $f$ maps every $(q-1,p)$-subspace of $D^{\mathrm{I}}_{q,p}$ into at least two $(q'-1,p')$-subspaces of $D^{\mathrm{I}}_{q',p'}$.
\end{enumerate}

Now, for any $x\in D_{q,p}$ we have $0\le n_0(x)\le q'-1$ and $\dim_{\mathbb C}\mathcal V_x=q'-n_0(x)-1$ is an integer, so $\{\dim_{\mathbb C}\mathcal V_x: x\in D_{q,p}\}\subset [0,q'-1]$ is a finite subset.
In particular, there exists $x'\in D_{q,p}$ such that $\dim_{\mathbb C}\mathcal V_{x'} = \min \{\dim_{\mathbb C}\mathcal V_x: x\in D_{q,p}\}=:k_0$.
Since $f$ is non-constant, $n_0(x')=\max\{n_0(x): x\in D_{q,p}\} \ge 1$ and thus $k_0=q'-n_0(x')-1\le q'-2$.
Then, $\dim_{\mathbb C}\phi^{-1}(x)=\dim_{\mathbb C}\mathcal V_x \ge k_0$ for all $x\in D_{q,p}$.
Thus, we have $D_{q,p}=\mathcal U_{k_0}$ and $\mathcal U_{k_0+1}\subset D_{q,p}$ is a proper complex-analytic subvariety.
In particular, $W:=D_{q,p}\smallsetminus \mathcal U_{k_0+1} \subset D_{q,p}$ is a dense open subset.
Now, we have a surjective holomorphic submersion $\widetilde{\phi}$ $:$ $\widetilde V$ $:=$ $V\cap (W \times D_{q',p'})$ $\to$ $W$ defined by $\widetilde{\phi}(x,z):=x$ for $(x,z)\in \widetilde V$, such that the fibers of $\widetilde{\phi}$ are biholomorphic to $\mathcal V_x \cong \mathbb P^{k_0}$ for all $x\in W$.
In particular, there is a local holomorphic section of $\widetilde{\phi}$ around any point $x_0\in W$.
More precisely, for any $x_0\in W$ there exists an open neighborhood $U_{x_0}$ of $x_0$ in $W$ and a holomorphic map $\widetilde g_{x_0}:U_{x_0}\to \widetilde V$ such that
\[ \widetilde{\phi}(\widetilde g_{x_0}(x))=x \quad \forall\;x\in U_{x_0}. \]
Note that $U_{x_0}\subset D_{q,p}$ is also open.
Thus, we can write $\widetilde g_{x_0}(x)=(x,g_{x_0}(x))$ for all $x\in U_{x_0}$, where $g_{x_0}:U_{x_0}\to D_{q',p'}$ is a holomorphic map.
By definition, we have
\[ f(x^\sharp) \subset (g_{x_0}(x))^\sharp \quad \forall\;x\in U_{x_0}. \]
In particular, $g_{x_0}:U_{x_0}\to D_{q',p'}$ is a local moduli map of $f$.
In addition, for any $x\in W$ we have $q'-n_0(x)-1=\dim_{\mathbb C} \mathcal V_x = k_0$ so that $n_0(x)=q'-k_0-1$, i.e., $f(x^\sharp)$ is contained in a unique $(q'-k_0-1,p')$-subspace of $D^{\mathrm{I}}_{q',p'}$.
In other words, since $W\subset D_{q,p}$ is a dense open subset, $f$ maps a general $(q-1,p)$-subspace of $D^{\mathrm{I}}_{q,p}$ into a unique $(q'-k_0-1,p')$-subspace of $D^{\mathrm{I}}_{q',p'}$.
We now summarize the above results as follows.

\begin{theorem}\label{thm:Behaviour_FIP_maps}
Let $f:D^{\mathrm{I}}_{q,p}\to D^{\mathrm{I}}_{q',p'}$ be a fibral-image-preserving non-constant holomorphic map with respect to the double fibrations
\[ D_{q,p} \xleftarrow{\pi_{q,p}^1} \mathbb P^{q-1}\times D^{\mathrm{I}}_{q,p}\xrightarrow{\pi_{q,p}^2} D^{\mathrm{I}}_{q,p}, \]
 \[ D_{q',p'} \xleftarrow{\pi_{q',p'}^1} \mathbb P^{q'-1}\times D^{\mathrm{I}}_{q',p'} \xrightarrow{\pi_{q',p'}^2} D^{\mathrm{I}}_{q',p'},\]
where $q,q'\ge 2$.
Then, either {\rm(1)} $f$ maps a general $(q-1,p)$-subspace of $D^{\mathrm{I}}_{q,p}$ into a unique $(q'-1,p')$-subspace of $D^{\mathrm{I}}_{q',p'}$ or {\rm(2)} $f$ maps every $(q-1,p)$-subspace of $D^{\mathrm{I}}_{q,p}$ into at least two $(q'-1,p')$-subspaces of $D^{\mathrm{I}}_{q',p'}$.

Let $n_0:= \max\{ n_0(x): x\in D_{q,p} \}$ and $n_0(x):=\min\{ n : f(x^\sharp) \subset D\cong D^{\mathrm{I}}_{n,p'}\subset D^{\mathrm{I}}_{q',p'}\}$, where $x_0^\sharp=\pi^2_{q,p}((\pi^1_{q,p})^{-1}(x_0))\cong D^{\mathrm{I}}_{q-1,p}$.
Write $k_0:=q'-n_0-1$.
Then, $W:=D_{q,p}\smallsetminus \mathcal U_{k_0+1}$ $\subset$ $D_{q,p}$ is a dense open subset and for any $x'\in W$, there exists a local moduli map $g_{x'}$ $:$ $U_{x'}$ $\subset$ $D_{q,p}$ $\to$ $D_{q',p'}$ of $f$, i.e.,
\[ f(x^\sharp)\subset (g_{x'}(x))^\sharp \quad \forall\;x\in U_{x'}, \]
where $U_{x'}\subset D_{q,p}$ is a connected open subset and
\[ \mathcal U_{l} = \left\{ x\in D_{q,p}: f(x^\sharp) \text{ lies in a $(q'-l-1,p')$-subspace of $D^{\mathrm{I}}_{q',p'}$} \right\} \]
for $l=0,1,2,\ldots,q'-1$.
In particular, $f$ maps every $(q-1,p)$-subspace of $D^{\mathrm{I}}_{q,p}$ into a $(q'-k_0-1,p')$-subspace of $D^{\mathrm{I}}_{q',p'}$, and $f$ maps a general $(q-1,p)$-subspace of $D^{\mathrm{I}}_{q,p}$ {\rm(}i.e., $x^\sharp$ for $x\in W${\rm)} into a unique $(q'-k_0-1,p')$-subspace of $D^{\mathrm{I}}_{q',p'}$.
\end{theorem}

Now, we observe that one can obtain moduli maps for a fibral-image-preserving holomorphic map by making use of another generalized double fibration introduced by Ng \cite[p.\,11]{Ng:2015}.

\begin{proposition}\label{prop:exist_global_moduli_map}
Let $f$ and $k_0$ be as in Theorem \ref{thm:Behaviour_FIP_maps}.
Write $l:=k_0+1$.
Then, there exist meromorphic maps $g_j:U \subset D_{q,p}\to D_{q',p'}$, $1\le j\le l$, such that $f(x^\sharp)\subset (g_j(x))^\sharp$ for all $x\in U$, i.e., $g_j:U\subset D_{q,p}\to D_{q',p'}$ is a moduli map of $f$ with respect to the double fibrations
\[ D_{q,p} \xleftarrow{\pi_{q,p}^1} \mathbb P^{q-1}\times D^{\mathrm{I}}_{q,p}\xrightarrow{\pi_{q,p}^2} D^{\mathrm{I}}_{q,p}, \]
 \[ D_{q',p'} \xleftarrow{\pi_{q',p'}^1} \mathbb P^{q'-1}\times D^{\mathrm{I}}_{q',p'} \xrightarrow{\pi_{q',p'}^2} D^{\mathrm{I}}_{q',p'},\]
where $U\subset D_{q,p}$ is a dense open subset.
\end{proposition}
\begin{proof}
By Theorem \ref{thm:Behaviour_FIP_maps}, $f$ maps every $(q-1,p)$-subspace into a $(q'-l,p')$-subspace, and $f$ maps a general $(q-1,p)$-subspace into a unique $(q'-l,p')$-subspace.
As in Ng \cite[p.\,11]{Ng:2015}, we may also consider the double fibration
\[ G(l,q'+p'-l) \supset D^l_{q',p'} \xleftarrow{\pi_l} G(l,q'-l) \times D^{\mathrm{I}}_{q',p'} \xrightarrow{\pi_{q'}} D^{\mathrm{I}}_{q',p'}, \]
where $\pi_l([X],[{\bf I}_{q'},Z']_{q'}) := [X,XZ']_{q'}$, $\pi_{q'}([X],[{\bf I}_{q'},Z']_{q'}):=[{\bf I}_{q'},Z']_{q'}$, and
\[ D^l_{q',p'}:= \left\{[W',W'']_{q'}\in G(l,q'+p'-l): W'\overline{W'}^T-W''\overline{W''}^T >0\right\} \]
with $W'\in M(l,q';\mathbb C)$ and $W''\in M(l,p';\mathbb C)$.
For any $[W',W'']_{q'}\in D^l_{q',p'}$, we have
\[\begin{split}
 [W',W'']_{q'}^\sharp:=&\pi_{q'}(\pi_l^{-1}([W',W'']_{q'}))
=\{Z\in D^{\mathrm{I}}_{q',p'}: W'Z=W''\} \\
\cong& D^{\mathrm{I}}_{q'-l,p'}. 
\end{split}\]
On the other hand, in this proof, for any $Z'\in D^{\mathrm{I}}_{q',p'}$ we write
\[ Z'^\sharp :=[{\bf I}_{q'},Z']_{q'}^\sharp 
= \left\{ [W',W'Z']_{q'} \in D^l_{q',p'}: [W']\in G(l,q'-l)\right\}
\cong G(l,q'-l). \]
For any $x\in D_{q,p}$ we define
\[ \mathcal V_x:=\left\{ W=[W',W'']_{q'} \in D^l_{q',p'}: f(x^\sharp) \subset W^\sharp \right\}
= \bigcap_{y\in f(x^\sharp)} y^\sharp, \]
which is a compact complex-analytic subvariety of $D^l_{q',p'}$.
Let
\[ \mathcal G:=\left\{(x,W)\in D_{q,p} \times D^l_{q',p'}: W\in \mathcal V_x\right\}. \]
Then, $\mathcal G$ is a complex-analytic variety.
Consider the projection map $\phi:=\mathrm{Pr}_1|_{\mathcal G}:\mathcal G \to D_{q,p}$ onto the first factor.
We can deduce from Theorem \ref{thm:Behaviour_FIP_maps} that 
\[ V:=\{ x\in D_{q,p}: \dim_{\mathbb C}\mathcal V_x \ge 1\} \]
is a proper complex-analytic subvariety of $D_{q,p}$.
Then, for any $x\in D_{q,p}\smallsetminus V$ we have $\mathcal V_x=\{g(x)\}$ for some point $g(x)\in D^l_{q',p'}$ and thus $\phi^{-1}(x)=\{x\}\times \mathcal V_x = \{(x,g(x))\}$.
Note that $D_{q,p}\smallsetminus V\subset D_{q,p}$ is a dense open subset .
Moreover, the set
\[ S:=\left\{(x,W)\in  \mathcal G: \dim_{(x,W)} \phi^{-1}(\phi(x,W)) \ge 1 \right\}=\phi^{-1}(V) \]
is a proper complex-analytic subvariety of $\mathcal G$, and we have a biholomorphism $\phi|_{\mathcal G\smallsetminus S}$ $:$ $ \mathcal G\smallsetminus S$ $\to$ $D_{q,p} \smallsetminus V$.
Note that 
\[ \{(x,g(x))\in \mathcal G: x\in D_{q,p} \smallsetminus V\}=\mathcal G\smallsetminus \mathrm{Pr}_1^{-1}(V) \subset (D_{q,p}\times D^l_{q',p'} )\smallsetminus \mathrm{Pr}_1^{-1}(V)\]
(resp.\,$\mathrm{Pr}_1^{-1}(V)\subset D_{q,p}\times D^l_{q',p'}$) is a complex-analytic subvariety.
Letting $\mathcal G'$ be the closure of $\mathcal G\smallsetminus \mathrm{Pr}_1^{-1}(V)$ in $D_{q,p}\times D^l_{q',p'}$, $\mathcal G'\subset D_{q,p}\times D^l_{q',p'}$ is a complex-analytic subvariety.
We now restrict $\mathrm{Pr}_1$ to the irreducible component of $\mathcal G'$ containing $\mathcal G\smallsetminus S$.
This yields a meromorphic map $g:D_{q,p} \to D^l_{q',p'} \subset G(l,q'+p'-l)$ defined on $D_{q,p}\smallsetminus V$ such that $\mathcal G\smallsetminus S$ is the graph of $g$ and
\[ f([A,B]_q^\sharp)\subset g([A,B]_q)^\sharp \cong D^{\mathrm{I}}_{q'-l,p'} \quad \forall \; [A,B]_q\in D_{q,p}\smallsetminus V=:U. \]
(Noting that $g:D_{q,p}\smallsetminus V \to D^l_{q',p'}$ is holomorphic.)

Writing $g=\begin{bmatrix}
g_1^T,\ldots,g_l^T
\end{bmatrix}^T$, where $g_j$, $1\le j\le l$, are row vectors of $g$, each $g_j$ yields a meromorphic map $g_j:D_{q,p}\to D_{q',p'}$ such that $f(x^\sharp)\subset (g_j(x))^\sharp$ for all $x\in U$.
More precisely, writing $g_j=[g_{j,1};g_{j,2}]_{q'}$, for any $[A,B]_q\in U$ we have
\[ g_{j,1}([A,B]_q)f(Z)=g_{j,2}([A,B]_q) \]
for all $Z\in [A,B]_q^\sharp$.
This shows that each $g_j:U\subset D_{q,p}\to D_{q',p'}$ is a moduli map of $f$.  
\end{proof}

\section{Case (2) of Conjecture \ref{Conj1}}
We will focus on Case (2) of Conjecture \ref{Conj1} in this section, namely, we study proper holomorphic maps from $D^{\mathrm{I}}_{p,q}$ to $D^{\mathrm{I}}_{p',q'}$, where $p\ge q\ge 2$ and $q'<\min\{p,2q-1\}$.

\subsection{Solution to Case (2) of Conjecture \ref{Conj1} when $q'=q+1$}\label{Section:corank_one}
Let $r,s,r'$ and $s'$ be positive integers such that $s'<\min\{r,2s-1\}$ and $r\ge s\ge 2$. If $s'-s=1$, then it is clear that the conditions $s'<\min\{r,2s-1\}$ and $r\ge s\ge 2$ become $3\le s\le r-2$.

By making use of Proposition \ref{Pro:RatProperM_FIP1}, we obtain a solution to Case (2) of Conjecture \ref{Conj1} under the additional assumption that $q'=q+1$, as follows.
\begin{theorem}\label{thm:proper4}
Let $f:D^{\mathrm{I}}_{p,q}\to D^{\mathrm{I}}_{p',q+1}$ be a proper holomorphic map, where $p,q$ and $p'$ are positive integers such that $3\le q\le p-2$.
Then, $p'\ge p$ and $f$ is of diagonal type.
\end{theorem}
\begin{proof}
Write $q':=q+1$.
Then, it follows from $3\le q\le p-2$ that $q'=q+1<2q-1$ and $q'=q+1 \le p-1 < p$, i.e., $q'<\min\{p,2q-1\}$.
Actually, we have seen that the condition $3\le q\le p-2$ is equivalent to the conditions $q+1=q'<\min\{p,2q-1\}$ and $p\ge q\ge 2$.
On the other hand, we have $p'\ge p$ by Proposition \ref{prop:proper2}.

By Proposition \ref{Pro:FIP1}, $f$ maps every $(p,q-1)$-subspace of $D^{\mathrm{I}}_{p,q}$ into a $(p',q'-1)$-subspace of $D^{\mathrm{I}}_{p',q'}$.
If $f$ maps every $(p,q-1)$-subspace of $D^{\mathrm{I}}_{p,q}$ into more than one $(p',q'-1)$-subspace of $D^{\mathrm{I}}_{p',q'}$, then $f$ maps every $(p,q-1)$-subspace of $D^{\mathrm{I}}_{p,q}$ into a $(p',q'-2)$-subspace $D^{\mathrm{I}}_{p',q'}$.
Note that we have $q'-2=q-1 \ge 2$ so that $\mathrm{rank}(D^{\mathrm{I}}_{p,q-1})=q-1=q'-2 = \mathrm{rank}(D^{\mathrm{I}}_{p',q'-2}) \ge 2$.
By Tsai \cite{Tsai:1993}, the restriction of $f$ to every $(p,q-1)$-subspace of $D^{\mathrm{I}}_{p,q}$ is standard. Therefore, $f$ itself is standard and thus of diagonal type. The proof is complete in this situation.
Now, we let 
\[ V:=\left\{[A,B]_q\in D_{q,p}: 
\begin{split}
& f^\dagger([A,B]_q^\sharp)\text{ is contained in more than}\\
&\text{one $(q'-1,p')$-subspace of $D^{\mathrm{I}}_{q',p'}$}
\end{split}\right\} \]
be a subset of $D_{q,p}$.
Then, $V^\sharp = \bigcup_{[A,B]_q\in V} [A,B]_q^\sharp \subseteq D^{\mathrm{I}}_{q,p}$ is a union of $(q-1,p)$-subspaces.
If $V\subset D_{q,p}$ has nonempty interior so that $V^\sharp$ contains a nonempty open subset $W$ of $D^{\mathrm{I}}_{q,p}$, then the above arguments imply that since $W=\bigcup_{[A,B]_q\in V} (W\cap [A,B]_q^\sharp)$ is open and $f^\dagger|_{W\cap [A,B]_q^\sharp}$ is standard for all $[A,B]_q\in V$ such that $W\cap [A,B]_q\neq \varnothing$, the whole map $f^\dagger$ is standard.
In particular, $f$ is standard. In this case, the proof is complete as before.

Now, we may suppose $V\subset D_{q,p}$ has empty interior so that $D_{q,p} \smallsetminus V$ is a dense subset.
Then, $f^\dagger$ maps a general $(q-1,p)$-subspace of $D^{\mathrm{I}}_{q,p}$ (i.e.,\,a $(q-1,p)$-subspace $[A,B]_q^\sharp$ with $[A,B]_q\in D_{q,p}\smallsetminus V$) into a unique $(q'-1,p')$-subspace of $D^{\mathrm{I}}_{q',p'}$.
Since $p>q'=q+1\ge 3$, by Ng-Tu-Yin \cite[Proposition 3.9]{NTY:2016} (see Proposition \ref{Pro:Pro3.9_NTY}), $f^\dagger$ has a moduli map $g$ $:$ $D_{q,p}$ $\dashrightarrow$ $D_{q',p'}$ $=$ $D_{q+1,p'}$ such that $g$ is a rational map, i.e., $g$ extends to a rational map $\hat g$ $:$ $\mathbb P^{q+p-1}$ $\to$ $\mathbb P^{q+p'}$, and $g$ is proper.

Let $U\subset D_{q,p}$ be the domain of $g$.
Then, $g:U\subset D_{q,p}\to D_{q',p'}=D_{q+1,p'}$ is a rational proper map which is fibral-image-preserving with respect to the double fibrations
\[ D_{q,p} \xleftarrow{\pi_{q,p}^1} \mathbb P^{q-1}\times D^{\mathrm{I}}_{q,p}\xrightarrow{\pi_{q,p}^2} D^{\mathrm{I}}_{q,p}, \]
\[ D_{q+1,p'} \xleftarrow{\pi_{q+1,p'}^1} \mathbb P^{q}\times D^{\mathrm{I}}_{q+1,p'} \xrightarrow{\pi_{q+1,p'}^2} D^{\mathrm{I}}_{q+1,p'}.\]
By Proposition \ref{Pro:RatProperM_FIP1}, $g$ is linear, i.e., $\deg(g)=1$.
Thus, $g$ is equivalent to the map
\[ D_{q,p}\ni [A;B]_q\mapsto [A,0;B,{\bf 0}]_{q+1}\in D_{q+1,p'} \]
so that $f^\dagger$ is equivalent to the map
\[ Z\mapsto \begin{bmatrix}
Z & {\bf 0}\\
{\bf 0} & h(Z)
\end{bmatrix} \]
for some holomorphic map $h:D^{\mathrm{I}}_{q,p}\to D^{\mathrm{I}}_{1,p'-p}$ by Lemma \ref{lem:Ch_di_type} and Seo \cite[Remark 4.3]{Seo:2015} (see Lemma \ref{lem:Seo15_Remark4.3}). In particular, $f^\dagger$ is of diagonal type, and so is $f$ by Lemma \ref{lem:trans2}.  
\end{proof}

We may now prove a special case of the main theorem, i.e., Theorem \ref{thm:MT}, when $q'=q+1$, as follows.

\begin{proof}[Proof of Theorem {\rm \ref{thm:MT}} when $q'=q+1$]
By Proposition \ref{prop:proper2}, we have $p'\ge p$ and $q'\ge q$.
Assume $q'=q+1$. Then, the rest follows directly from the proof of Theorem \ref{thm:proper4}. This finishes the proof of Theorem \ref{thm:MT} for the case where $q'=q+1$.  
\end{proof}

\subsection{Proof of Theorem \ref{thm:MT}}
We observe that results obtained in Section \ref{subsection:B_FIPM} are actually important to our study on proper holomorphic maps between type-$\mathrm{I}$ domains.
Based on these results, we can show the existence of global (meromorphic) moduli maps of proper holomorphic maps between certain type-$\mathrm{I}$ domains, as follows.

\begin{proposition}\label{pro:general_diag_type}
Let $f:D^{\mathrm{I}}_{q,p}\to D^{\mathrm{I}}_{q',p'}$ be a proper holomorphic map, where $p\ge q\ge 2$ and $p>q'\ge 3$.
Then, $f$ is fibral-image-preserving with respect to the double fibrations
\[ D_{q,p} \xleftarrow{\pi_{q,p}^1} \mathbb P^{q-1}\times D^{\mathrm{I}}_{q,p}\xrightarrow{\pi_{q,p}^2} D^{\mathrm{I}}_{q,p}, \]
 \[ D_{q',p'} \xleftarrow{\pi_{q',p'}^1} \mathbb P^{q'-1}\times D^{\mathrm{I}}_{q',p'} \xrightarrow{\pi_{q',p'}^2} D^{\mathrm{I}}_{q',p'}, \]
and there exists a rational proper map $g:U \subset D_{q,p}\to D_{q',p'}$ such that $f(x^\sharp)\subset (g(x))^\sharp$ for all $x\in U$, i.e., $g:U\subset D_{q,p}\to D_{q',p'}$ is a moduli map of $f$, where $U$ is a dense open subset.
If in addition that $q'<2q-1$, then $f$ is of diagonal type.
\end{proposition}
\begin{proof}
By Proposition \ref{prop:proper2}, $f$ maps every $(q-1,p)$-subspace into a $(q'-1,p')$-subspace so that $f$ is fibral-image-preserving with the given double fibrations.
Then, by Proposition \ref{prop:exist_global_moduli_map}, there exists a meromorphic map $g:U \subset D_{q,p}\to D_{q',p'}$ such that $f(x^\sharp)\subset (g(x))^\sharp$ for all $x\in U$, i.e., $g:U\subset D_{q,p}\to D_{q',p'}$ is a moduli map of $f$, where $U$ is a dense open subset.
By arguments in the proof of Proposition \ref{Pro:Pro3.9_NTY}, $g:D_{q,p}\dashrightarrow D_{q',p'}$ is indeed a rational proper map.

Assume $q'<2q-1$. By Proposition \ref{prop:proper2}, we have $p\le p'$ and $q\le q'$.
It then follows from Proposition \ref{pro:rat_proper_FIP_2} that $g$ is linear, i.e., $\deg(g)=1$, and $g$ is equivalent to the map
\[ [z_1,\ldots,z_q;z_{q+1},\ldots,z_{q+p}]_q\mapsto
[z_1,\ldots,z_q,{\bf 0};z_{q+1},\ldots,z_{q+p},{\bf 0}]_{q'}. \]
By the same arguments in the proof of Theorem \ref{thm:proper4}, $f$ is of diagonal type.  
\end{proof}
\begin{remark}
	Proposition \ref{pro:general_diag_type} actually gives a way to handle Case {\rm (1)} of Conjecture {\rm \ref{Conj1}}. 
	Let $g: D_{q,p}\to D_{q',p'}$ be a rational proper map which maps every maximal projective linear subspace contained in $D_{q,p}$ into a maximal projective linear subspace contained in $D_{q',p'}$, where $p\ge q\ge 2$, $q'<p$ and $p'<2p-1$.
	It is unknown if such a map $g$ is linear. Indeed, Case {\rm (1)} of Conjecture {\rm \ref{Conj1}} could be solved provided that one proved the linearity of such a map $g$.
\end{remark}

Now, we are ready to prove Theorem \ref{thm:MT}, as follows.

\begin{proof}[Proof of Theorem {\rm \ref{thm:MT}}]
Note that $D^{\mathrm{I}}_{p,q}$ is of rank $q$.
By Proposition \ref{prop:proper2}, we have $p\le p'$ and $q\le q'$ so that $D^{\mathrm{I}}_{p',q'}$ is of rank $q'$ since $q'<p$.
If $q'=2$, then we have $q=q'$ and thus $f$ is standard by Tsai \cite[Main Theorem]{Tsai:1993}. In particular, $f$ is of diagonal type when $q'=2$.
Now, we suppose $q'\ge 3$. Then, by Proposition \ref{pro:general_diag_type}, $f^\dagger$ is of diagonal type, and so is $f$ by Lemma \ref{lem:trans2}, as desired.
\end{proof}

Finally, we observe that Proposition \ref{pro:general_diag_type} and the known results also have the following implication.

\begin{proposition}\label{Prop:consequences}
Let $f:D^{\mathrm{I}}_{p,q}\to D^{\mathrm{I}}_{p',q'}$ be a proper holomorphic map, where $p\ge q\ge 2$ and $p>q'$.
Then, we have $q \le q'$ and $p\le p'$.
\end{proposition}
\begin{proof}
Since $f$ is proper, it is clear that $q\le \min\{p',q'\}\le q'$ by considering the rank of the type-$\mathrm{I}$ domains and using a result of Tsai \cite{Tsai:1993}.

Suppose $q'=2$. Then, we also have $q=2$ and thus $f$ has to be a standard embedding by Tsai \cite[Main Theorem]{Tsai:1993}. In particular, $f$ maps every $(1,p-1)$-subspace of $D^{\mathrm{I}}_{2,p}$ properly into a $(1,p'-1)$-subspace of $D^{\mathrm{I}}_{2,p'}$ by Tsai \cite{Tsai:1993}.
This induces a proper holomorphic map from $\mathbb B^{p-1}$ to $\mathbb B^{p'-1}$, and it follows that $p\le p'$.

Suppose $q'\ge 3$. Then, by Proposition \ref{pro:general_diag_type} there exists a rational proper map $g$ $:$ $U$ $\subset$ $ D_{q,p}$ $\to$ $D_{q',p'}$.
Assume the contrary that $p'<p$. Then, $\iota\circ g$ is a rational proper map from $D_{q,p}$ to $D_{q',p}$, where $\iota:D_{q',p'}\hookrightarrow D_{q',p}$ is the standard linear embedding, $\iota([A,B]_{q'}):=[A;B,{\bf 0}]_{q'}$.
By Baouendi-Huang \cite[Theorem 1.4]{Baouendi_Huang:2005}, $\iota\circ g$ is a standard linear embedding, and so is $g$.
But then this implies that $p\le p'$, a plain contradiction.
Hence, we have $p\le p'$, as desired.  
\end{proof}

\noindent\textbf{Acknowledgements}\\
The author would like to thank Professor Ngaiming Mok for invaluable comments and his encouragement during the preparation of this article. The author would also like to thank the referees for invaluable comments and for pointing out an error in the previous version of this article.

\end{document}